\documentclass[11pt]{article}
\usepackage{amssymb}
\usepackage{amsmath}
\usepackage{amsbsy}
\usepackage[latin1]{inputenc}
\usepackage{color}

\def\v        {{\boldsymbol v}}
\def\z        {{\boldsymbol z}}
\def\0        {{\boldsymbol 0}}

\def\h        {{\boldsymbol h}}

\def\u        {{\boldsymbol u}}
\def\v        {{\boldsymbol v}}
\def\w        {{\boldsymbol w}}

\def\f        {{\boldsymbol f}}

\def\b        {{\boldsymbol b}}
\def\H        {{\boldsymbol H}}
\def\F        {{\boldsymbol F}}
\def\G        {{\boldsymbol G}}
\def\V        {{\boldsymbol V}}
\def\L        {{\boldsymbol L}}

\def\B      {{\boldsymbol B}}

\def\W        {\boldsymbol W}

\newtheorem{theorem}{Theorem}%[section]
\newtheorem{lemma}[theorem]{Lemma}
\newtheorem{proposition}[theorem]{Proposition}

\newenvironment{proof}{ \textbf{Proof}:}{\hfill}

\begin{document}

\title{Periodic solution and asymptotic stability for the magnetohydrodynamic equations with inhomogeneous boundary condition}

\author{I. Kondrashuk
\thanks{Grupo de
Matem\'aticas Aplicadas, Dpto. de Ciencias B\'asicas, Facultad de
Ciencias, Universidad del B\'{\i}o-B\'{\i}o, Campus Fernando May,
Casilla 447, Chill\'an, Chile. E-mail: {\tt igor.kondrashuk@gmail.com} I. K. was supported by Fondecyt (Chile) Grants Nos. 1040368,
1050512 and 1121030, by DIUBB (Chile) Grants Nos. 102609 and 121909 GI/C-UBB.}, 
E.A. Notte-Cuello 
\thanks{
Dpto. de Matem\'{a}ticas. Universidad de La Serena, La Serena, Chile.
E-mail: \texttt{enotte@userena.cl} E. A. N-C. was supported in part by
Direcci\'{o}n de Investigaci\'{o}n de la Universidad de La Serena (DIULS)
through Grant No. PR12152.},
M. Poblete-Cantellano
\thanks{
Departamento de Matem\'{a}ticas, Universidad de Atacama, Avenida Copayapu
485, Casilla 240, Copiap\'{o}, Chile. E-mail: \texttt{mpoblete@mat.uda.cl}
M. P-C. was partially supported by Universidad de Atacama, project 221169
DIUDA 8/31.}
and 
M. A. Rojas-Medar
\thanks{
Instituto de Alta Investigaci\'{o}n, Universidad de Tarapac\'{a}, Casilla 7D, Arica, Chile. 
E-mail: \texttt{marko.medar@gmail.com} M.A. R-M was
partially supported by Project No. MTM2015-69875-P, by Ministerio de Ciencia
e Innovaci\'{o}n, Espa\~{n}a and Fondecyt (Chile) Grant No. 1120260.}}

\maketitle

\date{}

\begin{abstract}
We show, using the spectral Galerkin method together with compactness
arguments, existence and uniqueness of the periodic strong solutions for the
magnetohydrodynamics's type equations with inhomogeneous boundary
conditions. Also, we study the asymptotic stability  for time periodic
solution for this system. In particular, when the magnetic field $\h(x,t)$\
is zero, we obtain existence, uniqueness and asymptotic behavior of the strong
solutions to the Navier-Stokes equations with inhomogeneous boundary
conditions.
\end{abstract}

\noindent\thanks{2010 Mathematics Subject Classification: 35Q30, 35B10, 76W05} \\
\thanks{Keywords: Magnetohydrodynamic equations; periodic solutions.}

\section{Introduction}

From many decades is consolidated the awareness that the motion of incompressible electrical conducting
fluid can be modeled by the magnetohydrodynamic(MHD) equations, which
correspond to the Navier-Stokes (NS) equations coupled to the Maxwell
equations. This system of equations  plays an important role in various
applications, for example in phenomenons related to the plasma behavior \cite%
{alfven}, heat conductivity and nematic liquid crystal flows  \cite{Gala-1}-\cite{Gala-4}, 
stochastic dynamics \cite{Vasilev}. In the case when the MHD
equations have periodic boundary conditions these equations play an important
role in MHD generators \cite{Mitchner}. Also, these boundary conditions can
be considered in the tasks related with processes of the cooling nuclear reactors.

In presence of a free motion of heavy ions (see Schl\"uter $\cite{Schluter}, \cite{Schluter1} $ and Pikelner $\cite{Pikelner}),$ \ 
the MHD equation may be reduced to 
\begin{equation}
\begin{array}{lll}
\displaystyle\frac{\partial \u}{\partial t}-\displaystyle\frac{\eta }{{\rho }}\Delta 
\u+\u\cdot \nabla \u-\frac{\mu }{\rho }\h\cdot \nabla \h%
 & = & \displaystyle\f-\frac{1}{\rho }\nabla \left( p^{\ast }+\frac{\mu }{2}%
\h^{2}\right) \bigskip \\ 
\displaystyle\frac{\partial \h}{\partial t}-\frac{1}{{\bar{\mu} \sigma }}\Delta 
\h+\u\cdot \nabla \h-\h\cdot \nabla \u & = & -{\rm grad}%
w\bigskip \\ 
{\rm div}\,\u={\rm div}\,\h=0 &  & 
\end{array}
\label{Ch1}
\end{equation}%
with 
\begin{equation}
\u\left\vert _{\partial \Omega }\right. = \mbox{\boldmath $\beta$} _{1}(x,t)
,\qquad \h\left\vert _{\partial \Omega }\right. = \mbox{\boldmath $\beta$} _{2}(x,t).  \label{LS0}
\end{equation}

Here, $\u$ and $\h$ are  unknown velocity and magnetic field, respectively; 
$p^{\ast }$ is an unknown hydrostatic pressure; $w$ is
an unknown function related to the heavy ions (in such way that the density
of electric current, $j_{0},$ generated by this motion satisfies the
relation ${\rm rot}j_{0}=-\sigma \nabla w);$ $\rho$ is the density of mass of
the fluid (assumed to be a positive constant); $\bar{\mu} >0$ is a constant
magnetic permeability of the medium; $\sigma >0$ is a constant electric
conductivity; $\eta >0$ is a constant viscosity of the fluid; $\f$
is a given external force field. In this paragraph we used notations of \cite{NRR}. 
We should note the given external force field $\f$ is periodic throughout the paper.

As it has been mentioned in Ref. \cite{NRR}, several authors studied the initial value problem 
associated to the system (\ref{Ch1}).  By using the semigroup
results of Kato and Fujita \cite{Fujita}, Lassner  proved the existence and
uniqueness of strong solutions in Ref. \cite{Lassner}.  Then, Boldrini and Rojas-Medar \cite{Boldrini}, 
\cite{Rojas2} improved this result to global strong solutions by using the
spectral Galerkin method.  The regularity of weak solutions has been studied by 
Dam\'{a}zio and Rojas-Medar in \cite{Damasio}.  After this, Notte-Cuello and Rojas-Medar \cite{Notte}
used an iterative approach to show the existence and uniqueness of the strong
solutions. Later,  in works by Rojas-Medar and Beltr\'{a}n-Barrios \cite{Rojas1} and by Berselli and 
Ferreira \cite{Berselli} the initial value problem in  time dependent domains was considered. 

The periodic problem for the classical Navier-Stokes equations was studied
by Serrin \cite{Serrin} using the perturbation method and subsequently by
Kato \cite{Kato} using the spectral Galerkin method. Following the
methodology used by Kato, Notte-Cuello and Rojas-Medar \cite{NRR} studied
the existence and uniqueness of periodic strong solutions with homogeneous
boundary conditions for the MHD type equations. In this work it is considered
the periodic problem for the MHD equations with inhomogeneous boundary
conditions. We prove the existence and the uniqueness of the strong solutions to
this system of equations, following the methodology used by Morimoto \cite%
{Morimoto}, who presented results of existence and uniqueness of weak
solutions to the Navier-Stokes equations and to the Boussinesq equations.

On the other hand, Hsia et al \cite{Hsia} have shown that with the smallness assumption of the time periodic force, there exists 
only one time periodic solution to Navier-Sokes equations and this time periodic solution is globally asymptotically stable in the $\H^{1}$ sense. 
We follow the method used in \cite{Hsia} to perform a study of the asymptotic stability for our system.

\section{Preliminaries}

We begin by recalling definitions and facts  from Ref. \cite{NRR} to be used later in this paper. Let $\Omega$ be some bounded domain  in $\mathbb{R}^2$ or $\mathbb{R}^3$. 

The $L^{2}(\Omega )$-product and norm are denoted by $(,)~~\mbox{and}~~|~~|$%
, respectively; the $L^{p}(\Omega )$-norm by $|~~|_{L^{p}},1~\leq p\leq
\infty $; the $H^{m}(\Omega )$- norm is denoted by $\Vert ~~\Vert _{H^{m}}$
and the $W^{k,p}(\Omega )$-norm by $|~~|_{W^{k,p}}$.

Here $H^{m}(\Omega )=W^{m,2}(\Omega )$ and $W^{k,p}(\Omega )$ are  usual
Sobolev spaces, $H_{0}^{1}(\Omega )$ is the closure of $C_{0}^{\infty }(\Omega
)~\mbox{ in the}~H^{1}-\mbox{norm}$.

If B is a Banach space, we denote $L^{q}(0,T;B)$ the Banach space of the
B-valued functions defined in the interval (0, T) that are $L^{q}$%
-integrable in the sense of Bochner.

Let $C_{0,\sigma }^{\infty }(\Omega )=\{\v\in (C_{0}^{\infty
}(\Omega ))^{n};~\mbox{div}~\v=0\}$, $\H=\mbox{closure of}~~C_{0,\sigma }^{\infty }(\Omega )~\mbox{in}%
~\L^2(\Omega ))$, $\V=\mbox{closure
of}~~C_{0,\sigma }^{\infty }(\Omega )~\mbox{in}~\H_{0}^{1}(\Omega )$, 
$\H^1_{\sigma}(\Omega) = \{ \u \in \H^1(\Omega) : ~\mbox{div}~\u=0\}$.

Let $P$ be the orthogonal projection from $\L^{2}(\Omega )$~ onto $\H$
obtained by the usual Helmholtz decomposition. Then, the operator $%
A:\H\rightarrow \H$ given by $A=-P\Delta $ with domain $D(A)=\H^{2}(\Omega
)\cap ~\V$ is called the Stokes operator. 

In order to obtain regularity properties of the Stokes operator we will
assume that $\Omega $ is of class $C^{1,1}$ \cite{Amrouche}. This assumption
implies, in particular, that when $A\u\in \L^{2}(\Omega )$, then $%
\u\in \H^{2}(\Omega )~\mbox{and}~\Vert \u\Vert _{\H^{2}}~%
\mbox{and}~|A\u|$ are equivalent norms. 

The eigenfunctions and
eigenvalues of Stokes operator defined on $\V \cap \H^2(\Omega)$
are denoted by $\w^k$ and $\lambda_k$ respectively. It is well known that $\{ \w_k(x)\}^\infty_{k=1}$ form an
orthogonal complete system in the spaces $\H$, $\V$ and $\V \cap
\H^2(\Omega)$ equipped with the usual inner products
$(\u,\v), (\nabla \u,\nabla \v)$
and $(P\Delta \u,P\Delta \v)$ respectively.

Now, let us introduce some functions spaces consisting of $\tau $-periodic
functions. For $k\geq 0$, $k\in \mathbb{N}$, we denote by 
\[
C^{k}(\tau ;B)=\{f: \mathbb{R}\rightarrow B\mbox{ / }f\mbox{ is }\tau 
\mbox{-
periodic and }D_{t}^{i}f\in C(\mathbb{R};B)\mbox{  for any
}i\leq k\}. 
\]
Then, let us define the norm 
\[
\Vert f\Vert _{C^{k}(\tau ;B)}=\sup_{0\leq t\leq \tau }\sum_{i=1}^{k}\Vert
D_{t}^{i}f(t)\Vert _{B}. 
\]

We denote for $1\leq p\leq \infty $, the spaces 
\[
L^{p}(\tau ;B)=\{f: \mathbb{R}\rightarrow B\mbox{ / }f\mbox{ is measurable,}\,\tau %
\mbox{- periodic and }\Vert f\Vert _{L^{p}(\tau ;B)}<\infty \}, 
\]%
where 
\[
\Vert f\Vert _{L^{p}(\tau ;B)}=\left( \int_{0}^{\tau }\Vert f(t)\Vert
_{B}^{p}\right) ^{\frac{1}{p}}\mbox{  for  }1\leq p<\infty 
\]%
and 
\[
\Vert f\Vert _{L^{\infty }(\tau ;B)}=\sup_{0\leq t\leq \tau }\Vert f(t)\Vert
_{B}. 
\]%
Similarly, we denote by 
\[
W^{k,p}(\tau ;B)=\{f\in L^{p}(\tau ;B)\mbox{ / }D_{t}^{i}f\in L^{p}(\tau ;B)%
\mbox{  for any }i\leq k\}. 
\]%
In particular, $H^{k}(\tau ;B)=W^{k,2}(\tau ;B),$ when $B$ is a Hilbert
space.

The problem we consider is as follows: Let the given external force $\f$ be periodic in $t$ with some periodic $\tau .$ Then we try to prove the existence and uniqueness of periodic strong solutions $(\u,\h) $ of the magnetohydrodynamic equations (\ref{Ch1})-(\ref
{LS0}) with some periodic $\tau :$%
\begin{equation}
\u(x,t+\tau) =\u(x,t) ;\qquad 
\h(x,t+\tau) =\h(x,t) .  \label{Ch2}
\end{equation}

Now, according to the Gauss theorem, the boundary value $\mbox{\boldmath $\beta$} _{i}$ $i=1,2,$
should satisfy the so-called general outflow condition (GOC)%
\[
\left( GOC\right) \qquad \int_{_{\partial \Omega }}\mbox{\boldmath $\beta$} _{i}\cdot
nd\sigma =\sum\limits_{k=0}^{N}\int_{_{\Gamma _{k}}}\mbox{\boldmath $\beta$} _{i}\cdot
nd\sigma =0. 
\]%
If $N>1,$ the stringent outflow condition (SOC),

\[
\left( SOC\right) \qquad \int_{_{\Gamma _{k}}}\mbox{\boldmath $\beta$} _{i}\cdot nd\sigma
=0,\qquad ( k=0,1,...,N); 
\]%
is stronger than GOC.

In this work the following assumptions and results are considered,

\begin{description}
\item[$A_0$] $\Omega \subseteq \mathbb{R}^{n}, \,\,\, n=2,3$ bounded domain and $\partial \Omega $ 
consists of smooth $N+1$ connected components $\Gamma _{0},\Gamma _{1},...,\Gamma
_{N} $ and $\Omega $ being inside of $\Gamma _{0}$ $(N\geq 1),$   see ref. \cite{Morimoto}, p.1). 
This means $\Omega $ is enclosed by $%
\Gamma _{0},\Gamma _{1},...,\Gamma _{N},$ consequently. Such a structure of the boundary may be applied for the modeling of 
fluid movement inside of pipes.  The fluid velocity field is tangent to $\Gamma _{0}$ at the piece $\Gamma _{0}$ of the boundary.

\item[$A_1$] $\mbox{\boldmath $\beta$} _{i}\left( x,t\right) \in C^{1} ( \tau %
 ,\H^{1/2}\left(  \Omega \right) ) $ and satisfies $%
( SOC) ,$ $i=1,2.$ 

\end{description}

\begin{lemma} \label{lema1}
$\left[ \cite{Morimoto},p. 636\right] $ Suppose $\mbox{\boldmath $\beta$} \in
C^{1} ( \tau%
 ,\H^{1/2}\left(  \Omega \right) ) $ and satisfies (SOC). Then for every $\varepsilon >0,$ there
exists a solenoidal time-periodic function $\v\in C^{1}(\tau %
,\H_{\sigma }^{1}(\Omega)) $ such that%
\[
\begin{array}{cc}
\v(x,t) & =\mbox{\boldmath $\beta$} (x,t) ,\quad a.e.\,\, x\in
\partial \Omega ,\,\,\forall t\in \mathbb{R}, \\ 
&  \\ 
\left\vert ((\u\cdot \nabla) \v,\u)
\right\vert & \leq \varepsilon \left\vert \nabla \u\right\vert
^{2},\quad \forall \u\in \V,\forall t\in \mathbb{R}%
\end{array}%
\]
Moreover, if $\mbox{\boldmath $\beta$} \in
C^{1} ( \tau ,\W^{1,3/2}(\Omega) ) $ then $\v\in C^{1}(\tau ,\W^{2,2}(\Omega)) $.
\end{lemma}

\begin{proposition}
(Giga and Miyakawa \cite{Giga}). If $0\leq \delta <\frac{1}{2}+\frac{n}{4}$,
the following estimate is valid with a constant $C_{1}=C_{1}(\delta ,\theta
,\rho )$, 
\begin{equation}
|A^{-\delta }P\u\cdot \nabla \v|\leq C_{1}|A^{\theta }%
\u||A^{\rho }\v|\mbox{  for any  }\u\in D(A^{\theta
})\mbox{  and  }\v\in D(A^{\rho }),  \label{b1}
\end{equation}%
with $\theta ,\rho >0$  such that $\delta +\theta +\rho \geq \frac{n}{4}+\frac{1}{2}, \mbox{ }\rho +\delta >\frac{1}{2}.$
\end{proposition}

Also, we consider the Sobolev inequality \cite{Giga}, 
\[
|\u|_{L^{r}(\Omega )}\leq C_{2}\left\vert \u\right\vert _{%
\H^{\beta }},\mbox{ if
}\frac{1}{r}\geq \frac{1}{2}-\frac{\beta }{n}>0, 
\]%
and the inequality due to Giga and Miyakawa \cite{Giga} 
\begin{equation}
|\u|_{L^{r}(\Omega)}\leq C_{3}|A^{\gamma }\u|,\mbox{ if }%
\frac{1}{r}\geq \frac{1}{2}-\frac{2\gamma }{n}>0.  \label{b3}
\end{equation}%
Here, we note that if $r=n$ in (\ref{b3}) it follows 
\[
|\u|_{L^{n}(\Omega )}\leq C_{3}|A^{\gamma }\u|,\mbox{ with }%
\gamma =\frac{n}{4}-\frac{1}{2}. 
\]

\begin{lemma} \label{lema3} (Eq. (2.8)  in Kato \cite{Kato}) 
If $\u\in D(A^{\theta })$ and $0\leq \theta <\beta $, then

\[
|A^{\theta }\u(x)|\leq \mu ^{\theta -\beta }|A^{\beta }\u
(x)| 
\]%
where $\mu =\min \lambda _{j}>0,$ where $\{ \lambda_j\}_{j=1}^{\infty}$ are the eigenvalues of the Stokes operator.
\end{lemma}

\begin{lemma}
(Simon \cite{Simon})Let\ $X,B$\ and $Y$\ Banach spaces such that $%
X\hookrightarrow B\hookrightarrow Y$, where the first embedding is compact
and the second is continuous. Then, if $T>0$\ is finite, we have that the
following embedding is compact
\end{lemma}

\[
L^{\infty }(0,T;X)\cap \{\phi \mbox{    }\mbox{:}\mbox{    }\phi _{t}\in
L^{r}(0,T;Y)\}\hookrightarrow C(0,T;B)\mbox{,}\mbox{  if
}1<r\leq \infty . 
\]%

\section{Results}

Our results are the following.

\begin{theorem} 
(Existence) Suppose that $\Omega ,\mbox{\boldmath $\beta_i$}\,\, i=1,2 $ satisfy the
assumption $A_0$ and $A_1$ respectively and $\F, \G\in H^{1}(\tau ;\H)$ $(\tau >0)$. Then, there exists a constant $M >0$
such that if 
\[
\sup_{0\leq t\leq \tau }(\vert \F\vert _{\L^{n/2}(\Omega)}+ \vert \G\vert _{\L^{n/2}(\Omega) })\leq M
\]
the problem (\ref{Ch1})-(\ref{Ch2}) has a $\tau $-periodic strong solution 
$(\widetilde{\u}(t) ,\widetilde{\h}(t)) $ satisfying 
\[
(\widetilde{\u},\widetilde{\h}) \in ( 
H^{2}(\tau ;\H)) ^{2}\cap (H^{1}(\tau ;D(A))) ^{2}\cap (
L^{\infty }(\tau ;D(A))) ^{2}\cap (
W^{1,\infty }(\tau ;\V)) ^{2}, 
\] 
such that  $\widetilde{\u} =\u-B_{1}$ and  $\widetilde{\h} =\h-B_{2}$ 
for some $\tau $-periodic extension $B_{1}$ and $B_{2}$ of the boundary
values $\mbox{\boldmath $\beta$} _{1}$ and $\mbox{\boldmath $\beta$} _{2}$ respectively and $(\u,\h)$ satisfying the problem 
(\ref{Ch1})-(\ref{Ch2}). Here the functions  $\F$ and $\G$ are related to the
external force $\f$ and to the boundary data (see Eq. (14))%
\[
\begin{array}{l}
\F(t) = \displaystyle \alpha P\f(t) - \alpha\frac{d}{dt} B_1(t) + \nu A B_1(t) -\alpha P(B_1(t)\cdot\nabla B_1(t)) + P (B_2(t)\cdot\nabla B_2(t)),  \\
\G(t) = -\displaystyle \frac{d}{dt} B_2(t) + \chi A B_2(t) + P(B_2(t)\cdot\nabla B_1(t)) - P(B_1(t)\cdot\nabla B_2(t)).
\end{array} 
\]
\end{theorem}
 {\bf Remark:} As it follows from the proofs of Theorems 5 and 6 $M$ needs to be small. This implies that
$\mbox{\boldmath $\beta_i$\,\, i=1,2}$ and $\f$ must be small. 

\vspace{0.5cm}
\noindent {\bf Remark:} We observe that the hypothesis $F\in H^1(\tau; \H)$ implies in particular that $\frac{\partial B_i}{\partial t} \in H^1(\tau;\H)$ and $\Delta B_i \in H^1(\tau,\H)$, 
but Lemma \ref{lema1} only says that $B_i \in C^1(\tau; \W^{2,2}(\Omega).$ We believe that working as in \cite{ko} and \cite{Morimoto} it will be possible to show this regularity, however this requires a more detailed analysis, 
which we will not do in this article.

\begin{theorem}
(Uniqueness) The solution for (\ref{Ch1})-(\ref{Ch2}) given in the above
theorem is unique.
\end{theorem}

Now, we consider the initial-boundary value problem MHD
\begin{equation}
\label{PVI1} \left\{
\begin{array}{l}
{\displaystyle\frac{\partial{\u}}{\partial t}}\,+\,({\u}\cdot\nabla){\u}\, -\,\displaystyle\frac{\eta}{\rho}\Delta{\u}\,+\,\nabla\left(
p^* + \frac{\mu}{2}\h^2\right) \,=\,{\f}\qquad\,,
\\
\mbox{div }{\u}\,=\,0\qquad\mbox{in }Q_T\,,
\\
{\displaystyle\frac{\partial {\h}}{\partial t}}\, +\,({\u}\cdot\nabla)\,{\h}\,-\,({\h}\cdot\nabla)\,{\u}\, -\,\displaystyle\frac{1}{\bar{\mu}\sigma}\Delta{\h}\,
=\, {\rm grad}\, w \qquad\mbox{in }Q_T\,,
\end{array}
\right.
\end{equation}
\noindent with boundary and initial
conditions
\begin{equation}
\label{eqmariano2} \left\{
\begin{array}{l}
\u\left\vert _{\partial \Omega }\right. = \mbox{\boldmath $\beta$} _{1}(x,t) \\
\h\left\vert _{\partial \Omega }\right. = \mbox{\boldmath $\beta$} _{2}(x,t)\\
\u(x,0)=\u_{0}(x)\qquad\mbox{in }\Omega \,, \\
\w(x,0)=\w_{0}(x)\qquad\mbox{in }\Omega\,,
\end{array}
\right.
\end{equation}

The following result is a $\H^1$-stability result for the initial-value problem (\ref{PVI1})-(\ref{eqmariano2}) associated to the system (\ref{Ch1})- (\ref{LS0})
\begin{theorem}
Let $\F, \G\in H^{1}(\tau ;\H)$$(\tau >0),$ then there exist three
positives numbers $\gamma _{1},\gamma _{2}$ and $\gamma _{3}$ depending on
the viscosity coefficient $\nu $ and the size of the domain such that if $\F,\G$ satisfy%
\begin{equation}
\left\vert \F\right\vert _{L^{\infty }(0,\infty ;\L^{2}(\Omega
)^{2})}^{2}+ \left\vert \G\right\vert _{L^{\infty }(0,\infty ;\L^{2}(\Omega
)^{2})}^{2}\leq \gamma _{3},  \label{S1}
\end{equation}%
and $\left\{ \left( \u_{2}(t),\h_{2}
(t) \right) \right\} _{t\geq 0}$ is a strong solution of the system (\ref{Ch1})-(\ref{LS0})
with initial condition $(\u_{0},\h_{0}) $ satisfying 
\begin{equation}
\left \vert \u_{0}\right \vert _{\H^{1}}^{2}\leq \gamma _{1} ~~~\rm{and}~~~ \left \vert \h_{0}\right \vert _{\H^{1}}^{2}\leq \gamma _{2}
\label{S2}
\end{equation}%
and $\left\{(\u_{1}(t), \h_{1}(
t)) \right\} _{t\geq 0}$ is any other strong solution of (\ref{Ch1})-(\ref{LS0}), we have
\begin{equation}
\lim_{t\rightarrow \infty }\left\vert \u_{1}(t) -\u_{2}(t) \right\vert _{\H^{1}}^{2}=0  ~~~\rm{and}~~~
\lim_{t\rightarrow \infty }\left\vert \h_{1}(t) -\h_{2}(t) \right\vert _{\H^{1}}^{2}=0.  \label{S3}
\end{equation}%
The convergence rate in (\ref{S3}) is exponential.
\end{theorem}

A direct consequence of above theorem is the following. 

\begin{theorem}
Assume that $\F, \G\in H^{1}(\tau ;\H)$ $(\tau >0)$ and (\ref{S1}) hold
true, then for any two strong solution $(\u_{1}(t),
\h_{1}(t)) $ and $(\u_{2}(t), \h_{2}(t)) $ defined on the time
interval $[0,\infty) $ of the MHD equations (\ref{Ch1})-(\ref{LS0}), we have%
\begin{equation}
\lim_{t\rightarrow \infty }\left\vert \u_{1}(t) -\u_{2}(t) \right\vert _{\H^{1}}^{2}=0  ~~~\rm{and}~~~ 
\lim_{t\rightarrow \infty }\left\vert \h_{1}(t) -\h_{2}(t) \right\vert _{\H^{1}}^{2}=0.  \label{S4}
\end{equation}%
The convergence rate in (\ref{S4}) is exponential.
\end{theorem}

%{\bf Remark:}
%When the system (1)-(2) is not time-periodic,  it can be completed with the
%initial conditions 
%\[
%\mathbf{u}\left( x,0\right) =\mathbf{u}_{0}\left( x\right) \qquad \mathbf{h}%
%\left( x,0\right) =\mathbf{h}_{0}\left( x\right) ,\qquad x\in \Omega.  
%\]%
%However, when the system is time-periodic, what is our case, it must be
%completed with the initial conditions%
%\begin{equation}
%\mathbf{u}\left( x,0\right) =\mathbf{u}_{0}\left( x,\tau \right) \qquad 
%\mathbf{h}\left( x,0\right) =\mathbf{h}_{0}\left( x,\tau \right) ,\qquad
%x\in \Omega ,  \label{b4}
%\end{equation}%
%which are compatible with the condition (3) of the paper. Thus, when we refer to
%the initial conditions in theorem 7, we implicitly refer to the 
%conditions of type (\ref{b4}),  see, e.g. Refs. \cite{Blanca} and \cite{Morimoto}.  

Our main result is

\begin{theorem}[Stability]
Under the hypotheses of existence theorem, there exists a globally
asymptotically $\H^{1}$-stable time periodic strong solution $(\u,\h) $ to magnetohydrodynamic type equations (\ref{Ch1}). That is,
any other strong solution tends to this time-periodic solution $(\u,\h) $ asymptotically in the $\H^{1}$ sense.
\end{theorem}

\noindent  {\bf Remark:} With the periodic external force $\F, \G$ fixed, the previous result suggests that for any initial data 
$\v_0,\b_0 \in \V$, the unique strong solution obtained for $\v,\b$ tends to  unique strong periodic solution $\u,\h$ exponentially by a norm in $\H^1$.

\section{Approximate Problem and a priori estimates}

In this section we go along the lines of Ref. \cite{NRR} in which the homogeneous case was considered, using the spectral Galerkin method   
together with compactness arguments in order to prove  the existence and the uniqueness of the solution. 
The principal problem is to obtain the uniform boundedness of certain norms of $\u^{k}(t)$ and $\h^{k}(t)$
at some point $t^{\ast }$.  This difficulty was early treated by Heywood \cite%
{Heywood} to prove the regularity of the classical solutions for
Navier-Stokes equations.

The variables  $(\widetilde{\u}+B_{1},\widetilde{\h}+B_{2})$ satisfy the following equations:
\begin{equation}
\begin{array}{l}
\displaystyle\alpha \frac{\partial}{\partial t}( \widetilde{\u}+B_{1}) - \nu \Delta (\widetilde{\u}+B_{1}) +\alpha ( 
\widetilde{\u}+B_{1}) \cdot \nabla (\widetilde{\u
}+B_{1}) - (\widetilde{\h}+B_{2}) \cdot
\nabla (\widetilde{\h}+B_{2}) \medskip \\ 
=\displaystyle\alpha \f-\frac{1}{\mu }\nabla \left( p^{\ast }+\frac{\mu }{2}%
\left( \widetilde{\h}+B_{2}\right) ^{2}\right) \\ 
\\ 
\displaystyle\frac{\partial}{\partial t} (\widetilde{\h}+B_{2})
-\chi \Delta (\widetilde{\h}+B_{2}) + ( 
\widetilde{\u}+B_{1}) \cdot \nabla (\widetilde{%
\h}+B_{2}) - (\widetilde{\h}%
+B_{2}) \cdot \nabla (\widetilde{\u}+B_{1})
\medskip \\ 
=-{\rm grad}\,w.%
\end{array}
\label{m1}
\end{equation}

{\bf Remark 1:}
To ensure the periodicity of $B_{1}$ and $B_{2}$ we can see, for example, lemma 3.1 of Morimoto, ref.\cite{Morimoto} p. 636 of
the reference, we enunciated it in Lemma 1.

{\bf Remark 2:} In what follows we omit ``tilde'' over $\widetilde{\u}$ and $\widetilde{\h}.$ Instead, we will simple write $\u$ and $\h.$ This is done for the brevity of the following formulae. 

{\bf Remark 3:} We remind that the external force field $\f$ is $\tau-$periodic throughout all the paper.

Here we set $\alpha= \rho/\mu,$  $\nu = \eta/\mu$  and  $\chi=1/{\mu\sigma}.$
By putting $\widetilde{\u}=\u$ and $\widetilde{\h}=%
\h$ and rearranging terms, we obtain%
\begin{equation}
\begin{array}{l}
\displaystyle\alpha \frac{\partial \u}{\partial t}-\nu \Delta \u+\alpha 
\u\cdot \nabla \u- \h\cdot \nabla \h+\alpha \frac{%
\partial B_{1}}{\partial t}-\nu \Delta B_{1}+\alpha B_{1}\cdot \nabla B_{1}+%
\alpha\u\cdot \nabla B_{1}\medskip \\ 
\displaystyle + \alpha B_{1}\cdot \nabla \u - B_{2}\cdot \nabla \h - \h\cdot
\nabla B_{2} - B_{2}\cdot \nabla B_{2}= \alpha \f - \frac{1}{\mu }\nabla
\left( p^{\ast }+\frac{\mu }{2}\left( \h+B_{2}\right) ^{2}\right),
\\ 
\\ 
\displaystyle\frac{\partial \h}{\partial t} - \chi \Delta \h + \u\cdot \nabla
\h - \h\cdot \nabla \u + \frac{\partial B_{2}}{\partial t}-\chi \Delta
B_{2} + B_{1}\cdot \nabla \h - \h\cdot \nabla B_{1}
+ \alpha\u\cdot \nabla B_{2}\medskip \\ 
- \alpha B_{2}\cdot \nabla B_{1}-B_{2}\cdot \nabla \u + B_{1}\cdot
\nabla B_{2}=-{\rm grad}\,w.%
\end{array}
\label{m2}
\end{equation}

By using the operator $P,$ the periodic problem (\ref{Ch1})-(\ref{Ch2}) is
formulated as follows%
%\begin{equation}
%\begin{array}{ll}
%\displaystyle\alpha \frac{ d\u(t) }{ dt}+\nu A\u
%(t) + \alpha P(\u(t) \cdot \nabla 
%\u(t)) - P(\h(t) \cdot \nabla \h(t))  \medskip &  \\ 
% + P(\u(t) \cdot \nabla B_{1}(t)) + P( B_{1}(t) \cdot \nabla \u(t)) \medskip &  \\ 
%- P(B_{2}(t) \cdot \nabla \h(
%t)) - P(\h(t) \cdot \nabla B_{2}(
%t))  \bigskip = \F(t) &  \\ 
%\displaystyle\frac{d\h}{dt}+\chi A\h + P(\u
%(t) \cdot \nabla \h(t)) -
%P(\h(t) \cdot \nabla \u(t)
%) \medskip &  \\ 
%+ P(B_{1}(t) \cdot \nabla \h(
%t)) - P(\h(t) \cdot \nabla B_{1}(
%t)) + P(\u(t) \cdot
%\nabla B_{2}(t)) \medskip &  \\ 
% - P(B_{2}(t) \cdot \nabla \u(t)
%)  =\G(t), & 
%\end{array}
%\label{m2-5}
%\end{equation}
%
%\[
%\u(x,t+\tau) =\u(x,t) ;\qquad 
%\h(x,t+\tau) =\h(x,t),
%\]
%or, equivalently,
\begin{equation}\label{arica1}
\begin{array}{ll}
\displaystyle\alpha \frac{d}{dt} \u(t) +\nu A\u(t)+\alpha P(\u(t)\cdot\nabla\u(t)) -P(\h(t)\cdot\nabla\h(t)) + L_1 \u(t) + L_2\h(t) = \F(t), \\
\\
\displaystyle\frac{d}{dt}\h(t) +\chi A \h(t) +P(\u(t)\cdot\nabla\h(t)) - P(\h(t)\cdot\nabla\u(t)) + L_3\h(t) + L_4\u(t) =\G(t),
\end{array}
\end{equation}
\[
\u(x,t+\tau) =\u(x,t) ;\qquad 
\h(x,t+\tau) =\h(x,t),
\]
where 
\begin{equation}
\left\{\begin{array}{ll}
\displaystyle L_1\u(t) &= P(\u(t)\cdot\nabla B_1(t)) + P(B_1(t)\cdot\nabla\u(t)), \\
L_2\h(t) &= -P(\h(t)\cdot\nabla B_2(t))-P(B_2(t)\cdot\nabla\h(t)),  \\
\F(t) &= \displaystyle \alpha P\f(t) - \alpha\frac{d}{dt} B_1(t) + \nu A B_1(t) -\alpha P(B_1(t)\cdot\nabla B_1(t)) + P (B_2(t)\cdot\nabla B_2(t)),  \\
L_3\h(t) &= P(B_1(t) \cdot\nabla \h(t)) - P(\h(t)\cdot\nabla B_1(t)), \\
L_4 \u(t) &= -P(B_2(t) \cdot\nabla \u(t)) + P(\u(t)\cdot\nabla B_2(t)), \\
\G(t) &= -\displaystyle \frac{d}{dt} B_2(t) + \chi A B_2(t) + P(B_2(t)\cdot\nabla B_1(t)) - P(B_1(t)\cdot\nabla B_2(t)).
\end{array}\right.
\end{equation}

We consider $\V_{k}=span\{ \w_{1}(x),\w_{2}(x),...,\w_{k}(x)\}$ and the
approximations $\u^{k}(t)=\sum_{j=1}^{k}c_{jk}(t)\w_{j}(x)$ and $%
\h^{k}(t)=\sum_{j=1}^{k}d_{jk}(t)\w_{j}(x),$ of $\u$ and $%
\h,$ respectively, satisfying the following system of ordinary
differential equations. Here we reproduce the equations similar to Eq. (3.1) and Eq. (3.2) of \cite{NRR}, however, the terms with operators $L_1$ and $L_2$ 
are new in comparison with Eq. (3.1) and Eq. (3.2) of \cite{NRR} since these operators contain inhomogeneous boundary condition,

\begin{equation}
\begin{array}{l}
(\alpha \u_{t}^{k}+\nu A\u^{k}+\alpha P(\u
^{k}\cdot \nabla \u^{k}) - P(\h
^{k}\cdot \nabla \h^{k}) +L_1 \u^k + L_2\h^k,\w_{j}) = (\F,\w_j) \medskip \\ 
(\h_{t}^{k}+\chi A\h^{k} + P(\u
^{k}\cdot \nabla \h^{k}) - P(\h^{k}\cdot
\nabla \u^{k}) + L_3\h^k + L_4 \u^k,\w_{j}) = (\G,\w_j) \\ 
\u^{k}(x,t+\tau) =\u^{k}(x,t)
;\qquad \h^{k}(x,t+\tau) =\h^{k}(
x,t).%
\end{array}
\label{m3}
\end{equation}%
\[
\]

To show that system (\ref{m3}) has an unique $\tau -$periodic solution,
we consider the following linearized problem:%
\begin{equation}
\begin{array}{l}
(\alpha \u_{t}^{k}+\nu A\u^{k},\w_{j}) = (\F,\w_{j}) - (L_{1}\v^{k}
,\w_{j}) - (L_{2}\b^{k} ,\w_{j}) -\alpha(P(\v^k\cdot\nabla\v^k),\w_j) + (P(\b^k\cdot\nabla\b^k),\w_j) \medskip
\\ 
(\h_{t}^{k}+\chi A\h^{k},\w_{j}) = (\G,\w_j)  -(L_{3}\b^{k} ,\w_{j}) - (L_{4}\v^{k} ,\w_{j}) -(P(\v^k\cdot\nabla\b^k),\w_j) + (P(\b^k\cdot\nabla\v^k),\w_j)%
\end{array}
\label{m4}
\end{equation}%
where $\v^{k}(t)=\sum_{j=1}^{k}e_{jk}(t)\mbox{\boldmath $\omega$}_{j}(x)$ and $\b^{k}(t)=\sum_{j=1}^{k}g_{jk}(t)\mbox{\boldmath $\omega$}_{j}(x)$ are functions given in $%
C^{1}(\tau ;\V_{k}).$

It is well known that the linearized system (\ref{m4}) has an unique $\tau -$%
periodic solution $(\mathbf{u}^{k}(t),\mathbf{h}^{k}(t))\in (C^{1}(\tau
; \V_{k}))^{2}$ (see for instance, \cite{Amann}, \cite{Burton}). Consider the
map: $\Phi :(\v^{k},\b^{k})\rightarrow (\u^{k},\h^{k})$ in the
space $C^{0}(\tau ;\V_{k})\times C^{0}(\tau ;\V_{k})$. We shall show that $%
\Phi $ has a fixed point by Leray-Schauder Theorem.

We prove that for every $(\u^{k},\h^{k}) $ and $%
\lambda \in [0,1] $ satisfying $\lambda \Phi (\u
^{k},\h^{k}) =(\u^{k},\h^{k}) ,$%
\begin{equation}
\sup_{0\leq t\leq \tau }|\u^{k}(t)|\leq C\qquad \mbox{and \qquad }%
\sup_{0\leq t\leq \tau }|\h^{k}(t)|\leq C  \label{mk4/5}
\end{equation}%
where $C$ is a positive constant independent of $\lambda .$

For $\lambda =0,(\u^{k},\h^{k}) =(
0,0).$ Let $\lambda >0$ and assume that $\lambda \Phi ( \u
^{k},\h^{k}) =(\u^{k},\h^{k}).$
Then, from (\ref{m4}), we obtain
\begin{equation}
\begin{array}{l}
\displaystyle\frac{1}{2}\frac{d}{dt}\alpha \vert \u^{k}\vert
^{2}+\nu \vert \nabla \u^{k}\vert ^{2}=\lambda (
\alpha \F,\u^{k}) -\lambda (L_{1}\u^{k},\u^{k}) -\lambda (L_{2}\h^{k} ,\u^{k}) + \lambda(P(\h^k\cdot\nabla\h^k,\u^k),
\medskip \\ 
\displaystyle\frac{1}{2}\frac{d}{dt}\vert \h^{k}\vert ^{2}+\chi
\vert \nabla \h^{k}\vert ^{2}= \lambda(\G,\h^k)-\lambda (L_{3}\h^{k} ,\h^{k}) -\lambda (L_{4}\u^{k},\h^{k}) +\lambda(P(\h^k\cdot\nabla\u^k),\h^k) 
\end{array}
\label{mk5}
\end{equation}%
Summing the above equalities, we obtain%
\begin{eqnarray} \label{igor1}
&&\frac{1}{2}\frac{d}{dt}(\alpha \vert \u^{k}\vert
^{2}+ \vert \h^{k}\vert ^{2}) +\nu \vert \nabla 
\u^{k}\vert ^{2}+\chi \vert \nabla \h
^{k}\vert ^{2}\medskip  \nonumber \\
&=&\lambda (\F,\u^{k})+ \lambda(\G;\h^k) -\lambda (L_{1}\u^{k},\u^{k})
-\lambda (L_{2}\h^{k} ,\u^{k}) \medskip  \\
&&-\lambda (L_{3}\h^{k} ,\h^{k}) -\lambda (L_{4}\u^{k},\h^{k}) \medskip  \nonumber \\ 
&& + \lambda(P(\h^k\cdot\nabla\h^k),\u^k) + \lambda (P(\h^k\cdot\nabla\u^k,\h^k).
\nonumber
\end{eqnarray}

We observe that, since $\lambda \leq 1,$ we obtain%
\begin{equation}
\begin{array}{ll}
\lambda (\F,\u^{k}) & \leq \vert \F\vert \vert \nabla \u^{k}\vert, \medskip \\ 
\lambda (\G,\h^k) & \leq  \vert \G\vert \vert \nabla \h^{k}\vert. \\ 
\end{array}
\label{mk7}
\end{equation}%
Now, we use the Lemma \ref{lema1}, to obtain%
\begin{equation}
\begin{array}{ll}
-\lambda(L_1 \u^k,\u^k) = -\lambda(\u^k\cdot\nabla\B_1,\u^k) & \leq  \epsilon_1 \vert\nabla\u^k\vert^2, \\
-\lambda (L_2\h^k,\u^k) -\lambda(L_4\u^k,\h^k) = -\lambda (\h^k\cdot\nabla B_2,\u^k)-\lambda (\u^k \cdot\nabla B_2,\h^k) & \leq \epsilon_3 \vert \nabla\u^k\vert\vert\nabla\h^k\vert, \\
-\lambda (L_3\h^{k},\h^k) = (\h^k\cdot\nabla B_1,\h^k) & \leq \epsilon_2 \vert \nabla \h^k\vert^2.
\end{array}
\label{mk9}
\end{equation}%
Using the Young inequality, taking $\epsilon_1>0, \epsilon_2>0$ and $\epsilon_3>0$ suitable and summing the estimates (\ref{mk7}) and  (\ref{mk9}) together with the equality (\ref{igor1}), we have%
\begin{equation}
\begin{array}{l}
\displaystyle\frac{1}{2}\frac{d}{dt}(\alpha \vert \u^{k}\vert
^{2}+ \vert \h^{k}\vert ^{2}) +\nu \vert \nabla 
\u^{k}\vert ^{2}+\chi \vert \nabla \h
^{k}\vert ^{2}\medskip \\ 
\leq C\vert \F\vert ^{2}+ C\vert\G\vert^2.
\end{array}
\label{mk10}
\end{equation}%

Integrating in $t$ and using the periodicity of $(\u^{k},%
\h^{k}) $ we have%
\[
\int\nolimits_{0}^{\tau }\left( \nu \vert \nabla \u%
^{k}\vert ^{2}+\chi \vert \nabla \h^{k}\vert^{2}\right) dt\leq CM^{2} \tau, 
\]%
whence by the mean value theorem for integrals, there exists $t^{\ast }\in [
0,\tau] $ such that%
\begin{equation}
\nu\vert \nabla \u^{k}(t^{\ast }) \vert
^{2}+\chi \vert \nabla \h^{k}(t^{\ast })\vert ^{2}\leq CM^{2},  \label{mk11}
\end{equation}
$M$ is defined in Theorem 5.

On the other hand, by using the Lemma 3,\ with $\theta =0$ and $%
\beta =1/2,$%
\[
\vert \u^{k}(t^{\ast }) \vert \leq \mu
^{-1/2}\vert \nabla \u^{k}(t^{\ast })\vert 
\]%
and consequently%
\begin{equation}
\vert \u^{k}(t^{\ast })\vert ^{2}\leq \mu
^{-1}\vert \nabla \u^{k}(t^{\ast })\vert
^{2}\leq \frac{C}{\mu\nu}M^{2},  \label{N1}
\end{equation}%
analogously%
\begin{equation}
\vert \h^{k}(t^{\ast }) \vert ^{2}\leq \mu
^{-1}\vert \nabla \h^{k}(t^{\ast }) \vert
^{2}\leq \frac{C}{\mu\chi}M^{2}.  \label{N2}
\end{equation}

Finally, by integrating again (\ref{mk10}) from $t^{\ast }$ to $t+\tau ,$ with $t\in %
[0,\tau] ,$ we obtain (\ref{mk4/5}). As the map $\Phi $ is
continuous and compact in $C^{0}(\tau ;\V_{k}) $ we conclude the
existence of a fixed point $(\u^{k},\h^{k}) $
for $\Phi .$ Observe that (\ref{mk4/5}) holds for this $(\u
^{k},\h^{k}) .$

\begin{lemma} \label{lema7}
Let $(\u^{k}(t) ,\h^{k}(
t)) $ be the solution of (\ref{m3}). Suppose that%
\[
M<\min \left\{ \left( \frac{\nu}{P_{1}}\right) ^{2},\left( \frac{\chi }{P_{2}%
}\right) ^{2},1\right\}  
\]%
where%
\[
\begin{array}{ll}
P_{1} & =z \frac{\nu}{C} \mu ^{1-\gamma
}+C_{1}\alpha \frac{C}{\nu} \mu ^{\gamma
-3/2}+d_{5}+d_{4}\overline{C}\bigskip \\ 
& +2C_{1} \frac{C}{\chi} \mu ^{\gamma -3/2}%
\overline{C},\bigskip \\ 
P_{2} & =d_{3} \frac{\chi}{C} \mu ^{1-\gamma }+%
\widetilde{C}_{9} \frac{C}{\nu} \mu ^{\gamma
-3/2}+d_{6}+d_{4} \overline{C}\bigskip \\ 
& +2C_{1} \frac{C}{\chi} \mu ^{\gamma -3/2}%
\overline{C},
\end{array}%
\]%
then, we have%
\[
\vert A^{\gamma }\u^{k}(t) \vert
^{2}+ \vert A^{\gamma }\h^{k}(t) \vert
^{2}\leq E\mu ^{2\gamma -3}M   
\]
with $\gamma= \frac{n}{4}- \frac{1}{2}$.
\end{lemma}

\begin{proof}
The first part of the proof follows the proof of Lemma 2.1 of Ref.\cite{NRR}. Indeed, 
taking $A^{2\gamma }\u^{k}$ and $A^{2\gamma }\h^{k}$ as test functions in (\ref{m3}), we obtain

\begin{equation} \label{Cr1}
\begin{split}
\displaystyle &\frac{\alpha }{2}\frac{d}{dt}\vert A^{\gamma }\u
^{k}\vert ^{2}+\nu \vert A^{(1+2\gamma) /2}\u
^{k}\vert ^{2}=\medskip \\ 
&(\alpha \f (t) - \alpha P(\u^{k}\cdot
\nabla \u^{k}) + P(\h^{k}\cdot \nabla 
\h^{k}) - \alpha (B_{1})_{t}-\nu AB_{1},A^{2\gamma
}\u^{k}) \medskip \\ 
&-(\alpha P(B_{1}\cdot \nabla B_{1}) + P(\u
^{k}\cdot \nabla B_{1}) - P(B_{1}\cdot \nabla \u
^{k})  + P(B_{2}\cdot \nabla \h^{k})
,A^{2\gamma }\u^{k}) \medskip \\ 
&+ (P(\h^{k}\cdot \nabla B_{2}) + P(B_{2}\cdot
\nabla B_{2}) ,A^{2\gamma }\u^{k}),
\end{split}
\end{equation}

\begin{equation} \label{Cr2}
\begin{split}
\displaystyle &\frac{1}{2}\frac{d}{dt}\vert A^{\gamma }\h^{k}\vert
^{2}+\chi \vert A^{(1+2\gamma) /2}\h
^{k}\vert ^{2}=\medskip \\ 
&(- P(\u^{k}\cdot \nabla \h^{k}) 
+ P(\h^{k}\cdot \nabla \u^{k}) - (
B_{2}) _{t}-\chi AB_{2} - P(B_{1}\cdot \nabla \h
^{k}) ,A^{2\gamma }\h^{k}) \medskip \\ 
&+ (P(\h^{k}\cdot \nabla B_{1}^{k})
- P(\u^{k} \cdot \nabla B_{2})  - P(
B_{2}\cdot \nabla B_{1}) ,A^{2\gamma }\h^{k}) \medskip
\\ 
&- (P(B_{2}\cdot \nabla \u^{k}) - P(
B_{1}\cdot \nabla B_{2}) ,A^{2\gamma }\h^{k}).
\end{split}
\end{equation}

By using the Giga-Miyakawa estimate with $\theta =\gamma $ and $\rho
=(1+2\gamma) /2,$   we estimate terms in the right hand side of the above equalities as follows:
\[
\vert (\alpha \f(t) ,A^{2\gamma }\u
^{k}) \vert \leq \alpha \vert \f\vert
_{L^{n/2}}\vert A^{2\gamma }\u^{k}\vert _{L^{n/(
n-2) }}\leq \alpha \widehat{C}M\vert A^{(1+2\gamma)
/2}\u^{k}\vert,
\]%
here we use the H\"{o}lder's inequality
\begin{eqnarray*}
\vert (P\v\cdot \nabla \b ,A^{2\gamma }\mbox{\boldmath $\phi$}) \vert
&=&\vert (A^{\frac{2\gamma -1}{2}}P\v\cdot \nabla \b,A^{\frac{%
2\gamma +1}{2}}\mbox{\boldmath $\phi$} ) \vert \medskip \\
&\leq &C\vert A^{\gamma }\v\vert \vert A^{(1+2\gamma
) /2}\b\vert \vert A^{(1+2\gamma) /2}\mbox{\boldmath $\phi$}
\vert. 
\end{eqnarray*} 

In particular, the estimates of the right side of \eqref{Cr1} and \eqref{Cr2} may be done for each term. We take into account  
that $\Vert A^{2\gamma }\u\Vert \leq C\Vert A^{(2\gamma +1) /2}\u\Vert $ and estimate 
\[
\begin{array}{ll}
\vert (\alpha (B_{1}) _{t},A^{2\gamma }\u
^{k}) \vert & \leq \alpha C_{2}\vert (B_{1})
_{t}\vert \vert A^{(2\gamma +1) /2}\u
^{k}\vert, \bigskip \\ 
\vert (\nu AB_{1},A^{2\gamma }\u^{k}) \vert & 
\leq \vert ( \nu A^{\frac{2\gamma -1}{2}}AB_{1},A^{\frac{2\gamma +1%
}{2}}\u^{k}) \vert \medskip \\ 
& \leq \nu \overline{C_{3}}\vert AB_{1}\vert \vert A^{(
2\gamma +1) /2}\u^{k}\vert,%
\end{array}%
\]%
similarly%
\[
\begin{array}{ll}
\vert (\alpha P(B_{1}\cdot \nabla B_{1}) ,A^{2\gamma }%
\u^{k}) \vert & \leq \alpha C_{4}\vert A^{2\gamma
}B_{1}\vert \vert A^{(2\gamma +1) /2}B_{1}\vert
\vert A^{(2\gamma +1) /2}\u^{k}\vert, \bigskip
\\ 
\vert (P(\u^{k}\cdot \nabla B_{1})
,A^{2\gamma }\u^{k}) \vert & \leq C_{5}\vert
A^{3\gamma /2}B_{1}\vert \vert A^{(2\gamma +1) /2}%
\u^{k}\vert ^{2},\bigskip \\ 
\vert (P(B_{1}\cdot \nabla \u^{k})
,A^{2\gamma }\u^{k}) \vert & \leq C_{6}\vert
A^{\gamma }B_{1}\vert \vert A^{(2\gamma +1) /2}%
\u^{k}\vert ^{2},\bigskip \\ 
\vert (P( B_{2}\cdot \nabla \h^{k})
,A^{2\gamma }\u^{k}) \vert & \leq C_{7}\vert
A^{\gamma }B_{2}\vert \vert A^{(2\gamma +1) /2}%
\h^{k}\vert \vert A^{(2\gamma +1) /2}\u
^{k}\vert, \bigskip \\ 
\vert (P(\h^{k}\cdot \nabla B_{2})
,A^{2\gamma }\u^{k}) \vert & \leq C_{8}\vert
A^{3\gamma /2}B_{2}\vert \vert A^{(2\gamma +1) /2}%
\h^{k}\vert \vert A^{(2\gamma +1) /2}\u%
^{k}\vert, \bigskip \\ 
\vert (P(B_{2}\cdot \nabla B_{2}) ,A^{2\gamma }%
\u^{k}) \vert & \leq C_{9}\vert A^{\gamma
}B_{2}\vert \vert A^{(2\gamma +1) /2}B_{2}\vert
\vert A^{(2\gamma +1) /2}\u^{k}\vert.%
\end{array}%
\]%
Now, we bound the terms of \eqref{Cr2}% 
\begin{eqnarray*}
\vert (( B_{2}) _{t},A^{2\gamma }\h^{k})
\vert &\leq &\vert (A^{(2\gamma -1) /2}(
B_{2}) _{t},A^{(2\gamma +1) /2}\h^{k})
\vert \bigskip \\
&\leq &\widetilde{C_{1}}\Vert (B_{2}) _{t}\Vert
\vert A^{(2\gamma +1) /2}\h^{k}\vert,
\end{eqnarray*}

\begin{eqnarray*}
\vert \chi (AB_{2},A^{2\gamma }\h^{k}) \vert
&\leq &\vert (\chi A^{(2\gamma -1)
/2}AB_{2},A^{(2\gamma +1) /2}\h^{k}) \vert
\bigskip \\
&\leq &\widetilde{C_{2}}\vert A^{(2\gamma +1)
/2}B_{2}\vert \vert A^{(2\gamma +1) /2}\h
^{k}\vert,
\end{eqnarray*}%
\[
\begin{array}{ll}
\vert (P( B_{1}\cdot \nabla \h^{k})
,A^{2\gamma }\h^{k}) \vert & =\vert (
A^{(2\gamma -1) /2}P(B_{1}\cdot \nabla \h
^{k}) ,A^{(2\gamma +1) /2}\h^{k})
\vert \bigskip \\ 
& \leq C\vert A^{(2\gamma -1) /2}P(B_{1}\cdot
\nabla \h^{k}) \vert \vert A^{(2\gamma +1) /2}%
\h^{k}\vert \bigskip \\ 
& \leq \widetilde{C_{3}}\vert A^{\gamma }B_{1}\vert \vert
A^{(2\gamma +1) /2}\h^{k}\vert ^{2},%
\end{array}%
\]%
\[
\begin{array}{ll}
\vert (P( \h^{k}\cdot \nabla B_{1})
,A^{2\gamma }\h^{k}) \vert & = \vert (
A^{(2\gamma -1) /2}P(\h^{k}\cdot \nabla 
B_{1}) ,A^{(2\gamma +1) /2}\h^{k})
\vert \bigskip \\ 
& \leq C\vert A^{(2\gamma +1) /2}\h^{k}\vert
\vert A^{3\gamma /2}B_{1}\vert \vert A^{(2\gamma
+1) /2}\h^{k}\vert \bigskip \\ 
& \leq \widetilde{C_{4}}\vert A^{3\gamma /2}B_{1}\vert \vert
A^{(2\gamma +1) /2}\h^{k}\vert ^{2},%
\end{array}%
\]%
here we use $\theta =\frac{2\gamma +1}{2}$ and $\rho =\frac{3\gamma }{2}$ in
Giga-Miyakawa estimate,%
\[
\vert (P( \u^{k}\cdot \nabla B_{2})
,A^{2\gamma }\h^{k}) \vert \leq \widetilde{C_{5}}%
\vert A^{(2\gamma +1) /2}\u^{k}\vert
\vert A^{3\gamma /2}B_{2}\vert \vert A^{(2\gamma
+1) /2}\h^{k}\vert, 
\]%
\[
\begin{array}{cc}
\vert (P(B_{2}\cdot \nabla B_{1})
,A^{2\gamma }\h^{k}) \vert & = \vert ( A^{\frac{%
2\gamma -1}{2}}P(B_{2}\cdot \nabla B_{1}) ,A^{(
2\gamma +1) /2}\h^{k}) \vert \bigskip \\ 
& \leq \widetilde{C_{6}}\vert A^{\gamma }B_{2}\vert \vert
A^{(2\gamma +1) /2}B_{1}\vert \vert A^{(2\gamma
+1) /2}\h^{k}\vert,%
\end{array}%
\]%
\[
\vert (P(B_{2}\cdot \nabla \u^{k})
,A^{2\gamma }\h^{k}) \vert \leq \widetilde{C_{7}}%
\vert A^{\gamma }B_{2}\vert \vert A^{(2\gamma
+1) /2}\u^{k}\vert \vert A^{(2\gamma
+1) /2}\h^{k}\vert, 
\]%
\[
\vert (P( B_{1}\cdot \nabla B_{2})
,A^{2\gamma }\h^{k}) \vert \leq \widetilde{C_{8}}%
\vert A^{\gamma }B_{1}\vert \vert A^{(2\gamma
+1) /2}B_{2}\vert \vert A^{(2\gamma +1) /2}%
\h^{k}\vert . 
\]%
Now, summing the above estimates, we get%
\begin{equation}
\begin{array}{l}
\displaystyle\frac{\alpha }{2}\frac{d}{dt}\vert A^{\gamma }\u
^{k}\vert ^{2}+\frac{1}{2}\frac{d}{dt}\vert A^{\gamma }\h
^{k}\vert ^{2}+\nu \vert A^{\frac{1+2\gamma }{2}}\u
^{k}\vert ^{2}+\chi \vert A^{\frac{1+2\gamma }{2}}\h
^{k}\vert ^{2}\bigskip \\ 
\leq zM\vert A^{\frac{1+2\gamma }{2}}\u^{k}\vert
+M\vert A^{(2\gamma +1) /2}\h^{k}\vert
+2C_{1}\vert A^{\gamma }\h^{k}\vert \vert A^{(
2\gamma +1) /2}\h^{k}\vert \vert A^{\frac{2\gamma +1%
}{2}}\u^{k}\vert \bigskip \\ 
+M\vert A^{(2\gamma +1) /2}\h^{k}\vert
\vert A^{\frac{2\gamma +1}{2}}\u^{k}\vert +C_{1}\alpha
\vert A^{\gamma }\u^{k}\vert \vert A^{\frac{2\gamma
+1}{2}}\u^{k}\vert ^{2}+M\vert A^{\frac{2\gamma +1}{2}}%
\u^{k}\vert ^{2}\bigskip \\ 
+\widetilde{C_{9}}\vert A^{\gamma }\u^{k}\vert \vert
A^{(2\gamma +1) /2}\h^{k}\vert ^{2}M\vert
A^{(2\gamma +1) /2}\h^{k}\vert ^{2},%
\end{array}
\label{N3}
\end{equation}%
where we put%
\[
\begin{array}{l}
\alpha C_{2}\vert (B_{1}) _{t}\vert +\nu \overline{%
C_{3}}\vert AB_{1}\vert +\alpha C_{4}\vert A^{2\gamma
}B_{1}\vert \vert A^{{(2\gamma +1)
/2}}B_{1}\vert \bigskip \\ 
+C_{9}\vert A^{\gamma }B_{2}\vert \vert A^{{(2\gamma
+1) /2}}B_{2}\vert =d_{2}\leq M,\bigskip \\ 
\widetilde{C_{1}}\vert (B_{2})_{t}\vert +\widetilde{%
C_{2}}\vert A^{{(2\gamma +1) /2}}B_{2}\vert +%
\widetilde{C_{6}}\vert A^{\gamma }B_{2}\vert \vert
A^{{(2\gamma +1) /2}}B_{1}\vert \bigskip \\ 
+\widetilde{C_{8}}\vert A^{\gamma }B_{1}\vert \vert
A^{{(2\gamma +1) /2}}B_{2}\vert =d_{3}\leq M,%
\end{array}%
\]%
and 
\[
\begin{array}{l}
\\ 
C_{7}\vert A^{\gamma }B_{2}\vert +C_{8}\vert A^{3\gamma
/2}B_{2}\vert +\widetilde{C_{5}}\vert A^{3\gamma
/2}B_{2}\vert +\widetilde{C_{7}}\vert A^{\gamma }B_{2}\vert
=d_{4}\leq M,\bigskip \\ 
C_{5}\vert A^{3\gamma /2}B_{1}\vert +C_{6}\vert A^{\gamma
}B_{1}\vert =d_{5}\leq M,\bigskip \\ 
\widetilde{C_{3}}\vert A^{\gamma }B_{1}\vert +\widetilde{C_{4}}%
\vert A^{3\gamma /2}B_{1}\vert =d_{6}\leq M, \\ 
z=\alpha \widehat{C}+1.%
\end{array}%
\]%
 We should mention, that the constants that appear in right hand side of each estimation by  the Giga-Miyakawa inequalities are proper for the every inequality.
This is why we have so many constants. The presence of a such amount of constants in estimates reflects the difference with the homogeneous case of Ref.\cite{NRR}.

By using the Lemma \ref{lema3},\ with $\theta =0$ and $\beta =1/2$ we follow exactly the estimations done in Ref. \cite{NRR} for the proof of Lemma 2.1 and  obtain
\[
\vert A^{\gamma }\u^{k}(t^{\ast }) \vert
^{2}+\vert A^{\gamma }\h^{k}(t^{\ast }) \vert
^{2}\leq \left( \frac{1}{{\nu}^2}+\frac{1}{{\chi}^2}\right)C^2 \mu
^{2\gamma -3}M=E\mu ^{2\gamma -3}M. 
\]

Let $T^{\ast }=\sup \left\{ T/\left\vert A^{\gamma }\u^{k}\left(
t^{\ast }\right) \right\vert ^{2}+\left\vert A^{\gamma }\h^{k}\left(
t^{\ast }\right) \right\vert ^{2}\leq E\mu ^{2\gamma -3}M,\qquad t\in [
t^{\ast },T)\right\} .$ We will prove by contradiction that $T^{\ast
}=\infty .$ In fact, if $T^{\ast }$ is finite it should follow that $\forall
t\in [t^{\ast },T^{\ast }).$ Again, by following he proof of Lemma 2.1 in  Ref. \cite{NRR} we obtain  
\[
\vert A^{\gamma }\u^{k}(t^{\ast }) \vert
^{2}+\vert A^{\gamma }\h^{k}(t^{\ast }) \vert
^{2}\leq E\mu ^{2\gamma -3}M,\qquad t\in [t^{\ast },T). 
\]%
and 
\[
\vert A^{\gamma }\u^{k}(T^{\ast }) \vert
^{2}+ \vert A^{\gamma }\h^{k}(T^{\ast }) \vert
^{2}=E\mu ^{2\gamma -3}M, 
\]%
where $E= \left( \frac{1}{{\nu}^2} + \frac{1}{{\chi}^2} \right) C^2 .$ Therefore, for such a value $t=T^{\ast },$ we may estimate
\begin{eqnarray*}
zM\vert A^{(1+2\gamma) /2}\u^{k}\vert &\leq
&z \frac{\nu}{C} \mu ^{3/2-\gamma }\vert
A^{\gamma }\u^{k}\vert M^{1/2}\vert A^{(1+2\gamma
) /2}\u^{k}\vert \bigskip \\
&\leq &z  \frac{\nu}{C} \mu ^{1-\gamma
}M^{1/2}\vert A^{(1+2\gamma) /2}\u^{k}\vert
^{2}
\end{eqnarray*}%
where we use the inequality $\vert A^{\gamma }\u
^{k}\vert \leq \mu ^{-1/2}\vert A^{(1+2\gamma) /2}%
\u^{k}\vert .$ Similarly,%
\[
\begin{array}{ll}
d_{3}M\vert A^{(1+2\gamma) /2}\h^{k}\vert & 
\leq d_{3} \frac{\chi}{C} \mu ^{1-\gamma
}M^{1/2}\vert A^{(1+2\gamma) /2}\h^{k}\vert
^{2},\bigskip \\ 
C_{1}\alpha \left\vert A^{\gamma }\u^{k}\right\vert \left\vert
A^{\left( 1+2\gamma \right) /2}\u^{k}\right\vert ^{2} & \leq
C_{1}\alpha  \frac{C}{\nu} \mu ^{\gamma
-3/2}M^{1/2}\vert A^{(1+2\gamma) /2}\u
^{k}\vert ^{2},\bigskip \\ 
d_{5}M\vert A^{(1+2\gamma) /2}\u^{k}\vert
^{2} & \leq d_{5}M^{1/2}\vert A^{(1+2\gamma) /2}\u
^{k}\vert ^{2},\bigskip \\ 
\widetilde{C_{9}}\vert A^{\gamma }\u^{k}\vert \vert
A^{(1+2\gamma) /2}\h^{k}\vert ^{2} & \leq 
\widetilde{C}_{9} \frac{C}{\nu} \mu ^{\gamma
-3/2} M^{1/2}\vert A^{(1+2\gamma) /2}\h
^{k}\vert ^{2},\bigskip \\ 
d_{6}M\vert A^{(1+2\gamma) /2}\h^{k}\vert
^{2} & \leq d_{6}M^{1/2}\vert A^{(1+2\gamma) /2}\h
^{k}\vert ^{2},%
\end{array}%
\]%
and 
\[
d_{4}M\vert A^{(1+2\gamma) /2}\h^{k}\vert
\vert A^{(1+2\gamma) /2}\u^{k}\vert \leq
d_{4}M^{1/2}\overline{C}\left\{ \vert A^{(1+2\gamma) /2}%
\h^{k}\vert ^{2}+ \vert A^{(1+2\gamma) /2}%
\u^{k}\vert ^{2}\right\}, 
\]%
\[
\begin{array}{l}
2C_{1}\vert A^{\gamma }\h^{k}\vert \vert A^{(
1+2\gamma) /2}\h^{k}\vert \vert A^{(
1+2\gamma) /2}\u^{k}\vert \bigskip \\ 
\leq 2C_{1}  \frac{C}{\chi} \mu^{\gamma
-3/2}M^{1/2}\overline{C}\left\{ \vert A^{(1+2\gamma) /2}%
\h^{k}\vert ^{2}+ \vert A^{(1+2\gamma) /2}%
\u^{k}\vert ^{2}\right\} .%
\end{array}%
\]%
Consequently, the above estimate and (\ref{N3}) imply%
\[
\begin{array}{l}
\displaystyle\frac{\alpha }{2}\frac{d}{dt}\vert A^{\gamma }\u
^{k}\vert ^{2}+\frac{1}{2}\frac{d}{dt}\vert A^{\gamma }\h
^{k}\vert ^{2}+\nu \vert A^{\frac{1+2\gamma }{2}}\u
^{k}\vert ^{2}+\chi \vert A^{\frac{1+2\gamma }{2}}\h
^{k}\vert ^{2}\bigskip \\ 
\leq P_{1}M^{1/2}\vert A^{(1+2\gamma) /2}\u
^{k}\vert ^{2}+P_{2}M^{1/2}\vert A^{(1+2\gamma) /2}%
\h^{k}\vert ^{2},%
\end{array}%
\]%
where%
\[
\begin{array}{ll}
P_{1} & =z \frac{\nu}{C} \mu ^{1-\gamma
}+C_{1}\alpha \frac{C}{\nu} \mu ^{\gamma
-3/2}+d_{5}+d_{4}\overline{C}\bigskip \\ 
& +2C_{1} \frac{C}{\chi} \mu ^{\gamma -3/2}%
\overline{C},\bigskip \\ 
P_{2} & =d_{3} \frac{\chi}{C} \mu ^{1-\gamma }+%
\widetilde{C}_{9} \frac{C}{\nu} \mu ^{\gamma
-3/2}+d_{6}+d_{4} \overline{C}\bigskip \\ 
& +2C_{1} \frac{C}{\chi} \mu ^{\gamma -3/2}%
\overline{C},
\end{array}%
\]
Then, if $M<\min \left\{ \left( \frac{\nu}{P_{1}}\right) ^{2},\left( \frac{%
\chi }{P_{2}}\right) ^{2},1\right\} ,$ we have%
\[
\frac{\alpha }{2}\frac{d}{dt}\vert A^{\gamma }\u
^{k}\vert ^{2}+\frac{1}{2}\frac{d}{dt}\vert A^{\gamma }\h
^{k}\vert ^{2}<0,\qquad \mbox{at }t=T^{\ast }. 
\]%
Thus, in a neighborhood of $t=T^{\ast }$ it follows that
\[
\vert A^{\gamma }\u^{k}(t) \vert
^{2}+ \vert A^{\gamma }\h^{k}(t) \vert
^{2}\leq E\mu ^{2\gamma -3}M\qquad \mbox{for any }t\in [ T^{\ast
},T^{\ast }+\delta ). 
\]%
which implies $T^{\ast }=\infty .$ Then, we have%
\begin{eqnarray*}
\vert A^{\gamma }\u^{k}(t) \vert ^{2} &\leq
&E\mu ^{2\gamma -3}M\qquad \mbox{for any }t\in (-\infty ,\infty) \bigskip \\
\vert A^{\gamma }\h^{k}(t) \vert ^{2} &\leq
&E\mu ^{2\gamma -3}M\qquad \mbox{for any }t\in (-\infty ,\infty)
\end{eqnarray*}%
since $\u^{k}(t) $ and $\h^{k}(t)$
are periodical.
\end{proof}

\section{Estimates of the higher order derivatives}

In this section we derive estimates of derivatives of higher order. We need these  estimates  in order to show 
the convergence of the approximate solutions.  According to Lemma \ref{lema7}, for 
sufficiently small $M$ the approximate solutions satisfy 
\begin{equation}
\sup_{t}|A^{\gamma }\u^{k}(t)|\leq C_1(M),\mbox{      }%
\sup_{t}|A^{\gamma }\h^{k}(t)|\leq C_2(M)  \label{d1}
\end{equation}%
with $\gamma =\frac{n}{4}-\frac{1}{2}$, where $C_1(M)$ and $C_2(M)$ are constants
depending on $M$ and on a norm involving the border function $\mbox{\boldmath $\beta$}_i (x,t)$ and independent of $k$.
We may write a lemma, which is similar to Lemma 3.1 of Ref. \cite{NRR},

\begin{lemma}
Let $(\u^{k}(t),\h^{k}(t))$ be the solution of (\ref{m3})
given above. Set 
\[
M_{0}=\left(\int_{0}^{\tau }(|\F(t)|^{2}+ \vert \G(t)\vert^2)dt\right)^{\frac{1}{2}},\mbox{      }%
M_{1}=\left(\int_{0}^{\tau }|(\F_{t}(t)|^{2}+\vert \G_t(t)\vert^2 dt\right)^{\frac{1}{2}}. 
\]
\end{lemma} 

Then, we have 
\[
\sup_{0\leq t\leq \tau }|\nabla \u^{k}(t)|^{2}\leq
C(M_{0},M),\mbox{
}\sup_{0\leq t\leq \tau }|\nabla \h^{k}(t)|^{2}\leq
C(M_{0},M), 
\]%
and 
\[
\sup_{t}(\alpha |\u^{k}_t(t)|^{2}+|\h_{t}^{k}(t)|^{2})\leq
C(M_{0,}M_{1},M), 
\]%
where $C(M_{0},M)$ and $C(M_{0},M_{1},M)$ denote constants depending on $M_{0},M_{1}$ are independent of $k.$

\textbf{Proof. }  We repeat here the trick with test functions used by us in the proof of Lemma 1.   

Taking $A\u^{k}$ and $A\h^{k}$ as test
functions in (\ref{m3}), we get%

\begin{eqnarray*}
\left( \alpha \u_{t}^{k}+\nu A\u^{k},A\u^{k}\right) &=& ( \F -\alpha P( \u^{k}\cdot\nabla \u^{k}) + P( \h^{k}\cdot \nabla \h^{k}),A\u^{k})  \\ 
&&+( L_1(\u^k) ,A\u^{k}) + ( L_2(\h^k) ,A\u^{k}), \bigskip \\ 
(\h_{t}^{k}+\chi A\h^{k},A\h^{k})  &=&
( \G- P(\u^{k}\cdot \nabla \h^{k}) 
+ P(\h^{k}\cdot \nabla \u^{k}) ,A\h^{k})\\
&&+( L_3(\h^k) ,A\h^{k}) + ( L_4(\u^k) ,A\h^{k}),
\end{eqnarray*}%

Then, we follow the same lines that we did in the proof of Lemma 3.1 of Ref. \cite{NRR}, recalling the estimate (\ref{d1}) are sufficiently small  (if $M$ is small) and by hypotheses $\vert AB_i\vert$ and $\vert A^\gamma B_i\vert$ ($i=1,2$) also are sufficiently small we can obtain the following inequality

\begin{eqnarray}\label{d3}
\frac{d}{dt} \left(\alpha \vert \u^k\vert^2 + \vert \nabla\h^k\vert^2 \right) + 2\nu \vert A\u^k\vert^2 + 2 \chi\vert A\h^k\vert^2 \leq C
\end{eqnarray}
where the constant $C>0$ depends on $\partial \Omega$, $B_i, i=1,2$, $M$, $\f$.

 Integrating (\ref{d3}) and recalling the periodicity of $\nabla \u^{k}(t)$ and $\nabla \h^{k}(t) ,$ we have%
\begin{eqnarray*}
\displaystyle\int_{0}^{\tau }(2\nu\vert A\u^{k}\vert
^{2}+2\chi \vert A\h^{k}\vert ^{2}) dt \leq D_1
\end{eqnarray*}
where $D_1\geq C\tau$.
%or
%\begin{equation}
%\begin{array}{l}
%\displaystyle\int\nolimits_{0}^{\tau }((\nu-N_{1}(M))
%\vert A\u^{k}\vert ^{2}+ (\chi -N_{2}(M)) \vert A\h^{k}\vert ^{2}) dt\bigskip \\ 
%\displaystyle\leq \widehat{M}_{1}(M,J_{1B}) \left( \int\nolimits_{0}^{\tau }\vert
%A\u^{k}\vert ^{2}\right) ^{1/2}+\widehat{M}_{2}(J_{2B})
%\left( \int\nolimits_{0}^{\tau }\vert A\h^{k}\vert
%^{2}\right) ^{1/2}.%
%\end{array}
%\label{d4}
%\end{equation}
%Then, if $d=\min \left\{ \left(\nu-N_{1}\left( M\right) \right) ,\left( \chi
%-N_{2}\left( M\right) \right) \right\} >0,$ from (\ref{d4}) we have,%
%\[
%\begin{array}{l}
%\displaystyle\int\nolimits_{0}^{\tau }(\vert A\u^{k}\vert
%^{2}+\vert A\h^{k}\vert ^{2}) dt\medskip \\ 
%\displaystyle\leq d^{-1}\widehat{M}_{1}(M,J_{1B}) \left( \int\nolimits_{0}^{\tau
%}\vert A\u^{k}\vert ^{2} \right) ^{1/2}+d^{-1}\widehat{M}_{2}(
%J_{2B}) (\int\nolimits_{0}^{\tau }\vert A\h
%^{k}\vert ^{2}) ^{1/2}\medskip \\ 
%\displaystyle\leq \frac{( d^{-1}\widehat{M}_{1}( M,J_{1B})) ^{2}}{2}+\frac{%
%1}{2}\int\nolimits_{0}^{\tau }\vert A\u^{k}\vert ^{2}+%
%\frac{( d^{-1}\widehat{M}_{2}(J_{2B})) ^{2}}{2}+\frac{1}{2}%
%\int\nolimits_{0}^{\tau }\vert A\h^{k}\vert ^{2},%
%\end{array}%
%\]%
%where we have used the Young's inequality, then%
%\[
%\begin{array}{ll}
%\displaystyle\int\nolimits_{0}^{\tau }(\vert A\u^{k}\vert
%^{2}+ \vert A\h^{k}\vert ^{2}) dt & \leq (
%d^{-1}\widehat{M}_{1}(M,J_{1B})) ^{2}+ (d^{-1}\widehat{M}_{2}(
%J_{2B})) ^{2}\medskip \\ 
%& \leq D(M,J_{1B},J_{2B}).
%\end{array}%
%\]

Finally, applying the Mean Value Theorem for integrals, we have that there
exists $t^{\ast }\in [ 0,\tau ]$ such that 
\[
|A\u^{k}(t^{\ast })|^{2}+|A\h^{k}(t^{\ast })|^{2}\leq \tau
^{-1}D. 
\]%
By using the Lemma \ref{lema3} , with $\theta =\frac{1}{2}$ , $\beta =1,$ we have 
\[
|\nabla \u^{k}(t^{\ast })|^{2}\leq \mu ^{-1}|A\u^{k}(t^{\ast
})|^{2}\leq \mu ^{-1}\tau ^{-1}D
\]%
and 
\[
|\nabla \h^{k}(t^{\ast })|^{2}\leq \mu ^{-1}|A\h^{k}(t^{\ast
})|^{2}\leq \mu ^{-1}\tau ^{-1}D. 
\]%
Now, integrating inequality (\ref{d3}) from $t^{\ast }$ to $t+\tau $ $%
(t\in [0,\tau])$, we deduce easily 
\begin{equation}
\sup_{t}|\nabla \u^{k}(t)|\leq C(M_0, M) ,\quad
\sup_{t}|\nabla \h^{k}(t)|\leq C(M_0,M)
\label{d2a}
\end{equation}%
where $C(M_0,M)$ is independent of $k.$

Similarly, taking $\u_{t}^{k}$ and $\h_{t}^{k}$ as test
functions in (\ref{m3}), we can show that%
\[
\sup_{t}|\u_{t}^{k}(t)|\leq C(M_{0},M_1, M)
,\quad \sup_{t}|\h_{t}^{k}(t)|\leq D(M_{0},M_1,M). 
\]
This completes the proof of lemma.

The proof of the following lemma is omitted, since it is similar to the
proofs of the previous lemmas and one can follow the methodology of Lemma
3.2 of \cite{NRR}.

\begin{lemma}
Let $(\u^{k}(t),\h^{k}(t))$ be the approximate solution of (%
\ref{m3}) given above. Then, we have 
\[
\sup_{t}|A\u^{k}(t)|\leq C(M_{0},M_{1},M),\quad
\sup_{t}|A\h^{k}(t)|\leq C(M_{0},M_{1},M) 
\]

\[
\int_{0}^{\tau }(|A\u_{t}^{k}(t)|^{2}+|A\h
_{t}^{k}(t)|^{2})dt\leq C(M_{0,}M_{1},M), 
\]%
\[
\int_{0}^{\tau }(|\u_{tt}^{k}(t)|^{2}+|\h
_{tt}^{k}(t)|^{2})dt\leq C(M_{0,}M_{1},M).
\]
\end{lemma}

\section{Proof of Theorem 5 and Theorem 6}

In this section we partially use a similar strategy to prove uniqueness and existence theorems that was applied in Ref. \cite{NRR} to the case of homogeneous boundary condition.
First, we prove Theorem 5.  By the Aubin-Lions theorem, it follows from  estimates (\ref{mk4/5}) that there are subsequences $\u^{k}(t)$ and $\h^{k}(t)$ such that

\[
\u^{k}\rightarrow \u\mbox{, }\h^{k}\rightarrow \h
\mbox{, strongly in  }L^{\infty }(\tau ;\V). 
\]

We may write by using Lemma 12 
\begin{eqnarray*}
\u^{k}\rightarrow \u \mbox{, }\h^{k}\rightarrow \h
\mbox{, }w^{\ast }\mbox{ in  }L^{\infty }(\tau ;D(A)), \\ 
\u_{t}^{k}\rightarrow \u_{t}\mbox{, }\h
_{t}^{k}\rightarrow \h_{t}\mbox{, }w^{\ast }\mbox{ in
}L^{\infty }(\tau ;\V),  
\end{eqnarray*}
in which the functions $\u(t)$ and $\h(t)$ satisfy 
\[
\u,\h\in H^{2}(\tau ;\H)\cap H
^{1}(\tau ;D(A))\cap L^{\infty }(\tau ;D(A))\cap W^{1,\infty }(\tau ;\V). 
\]
Our aim is to show that
\[
\u_{t}^{k}\rightarrow \u_{t}\mbox{, }\h
_{t}^{k}\rightarrow \h_{t}\mbox{,
strongly in  }L^{\infty }(\tau ;\H). 
\]

We may take $\phi =\u_{t}$ and $\phi =\h_{t}$ in Lemma $4$, with $X=\V,Y=B=\H.$  In such way we establish the desired convergences. 
After the establishing of these convergences, we take the limit along the previous subsequences in (\ref{m3}), 
and we  conclude that $(\u,\h)$ is a periodic strong solution of (\ref{Ch1})-(\ref{Ch2}). 
This proves Theorem 5 dedicated to existence of periodic solution.

To prove Theorem 6 dedicated to the uniqueness, we consider that $(\u_{1},\h_{1})$
and $(\u_{2},\h_{2})$ are two solutions of problem (\ref{Ch1}%
)- (\ref{Ch2}). By defining the differences

\[
\w=\u_{1}-\u_{2}\mbox{, }\z=\h_{1}-\h_{2},
\]

we have from (\ref{arica1})   
 \begin{eqnarray} \label{LS1} 
\displaystyle\alpha \frac{ d\w}{dt }+\nu A\w &=& -\alpha P\w\cdot \nabla \u
_{1} - \alpha P\u_{2}\cdot \nabla \w+P\z\cdot \nabla \h
_{1}  +P\h_{2}\cdot \nabla \z -  L_1(\w)- L_2(\z),  \nonumber \\
\displaystyle\frac{d \z}{dt}+\chi A\z &=&-P\w\cdot \nabla \h_{1} -
P\u_{2}\cdot \nabla \z+P\z\cdot \nabla \u_{1} 
+P\h_{2}\cdot \nabla \w- L_3(\z)- L_4(\w) ,%  
\end{eqnarray}  

Then, by multiplying the first equation of (\ref{LS1}) (respectively the second equation of (\ref{LS1})) by $\w$ (respectively by $\z$) and 
integrating on $\Omega$, we obtain repeating mainly the approach used in Section 5 of Ref.\cite{NRR} 

\[
\begin{array}{l}
\displaystyle\frac{1}{2}\frac{d}{dt}(\alpha |\w|^{2}+|\z|^{2})+\nu |\nabla \w|^{2}+\chi
|\nabla \z|^{2}\medskip \\ 
=\alpha (P\w\cdot \nabla \w,\u_{1}) - (P\z\cdot \nabla
\w,\h_{1}) + (P\w\cdot \nabla \w,B_{1}) - (
P\z\cdot \nabla \w,B_{2}) \medskip \\ 
+ (P\w\cdot \nabla \z,\h_{1}) - (P\z\cdot \nabla \z,%
\u_{1}) - (P\z\cdot \nabla \z,B_{1}) + (P\w\cdot
\nabla \z,B_{2}) .%
\end{array}%
\]
Now, by Giga-Miyakawa $(\vert A^{-\delta }P\u\cdot \nabla
\v\vert \leq C_{1}\vert A^{\theta }\u\vert
\vert A^{\rho }\v\vert) $\ with $\delta =\gamma $ and $%
\theta =\rho =1/2,$ we have, repeating the approach used in Section 5 of Ref.\cite{NRR}

\begin{eqnarray*}
\vert \alpha (P\w\cdot \nabla \w,\u_{1}) \vert \leq C_{1}\vert \nabla \w\vert ^{2}\vert A^{\gamma } \u_{1}\vert \leq C_{1}C(M) \vert \nabla \w\vert^{2}, \\ 
\vert (P\z\cdot \nabla \w,\h_{1}) \vert \leq  C_{1}C(M) \vert \nabla \z\vert \vert\nabla \w\vert \leq \frac{C_{1}C(M) }{2}\vert \nabla\z\vert ^{2}+\frac{C_{1}C(M) }{2}\vert \nabla \w\vert ^{2},
\end{eqnarray*}

Similarly, we may evaluate $\left\vert (P\w\cdot \nabla \w,B_{1}) \right\vert,$ $\left\vert (P\w\cdot \nabla \w,B_{2}) \right\vert,$ 
$\left\vert (P\z\cdot \nabla \w,B_{2}) \right\vert,$  $\left\vert (P\w\cdot \nabla \z,\h_{1}) \right\vert,$  $\left\vert \left( P\z\cdot \nabla \z,\u_{1}\right) \right\vert,$  
$\left\vert \left( P\z\cdot \nabla \z,B_{1}\right) \right\vert,$  $\left\vert \left( P\w\cdot \nabla \z,B_{2}\right) \right\vert.$  

Then, by using the estimates above we have%
\[
\frac{1}{2}\frac{d}{dt}(\alpha |\w|^{2}+|\z|^{2})+\nu |\nabla \w|^{2}+\chi
|\nabla \z|^{2}\leq D(M)( \nu |\nabla \w|^{2}+\chi |\nabla
\z|^{2}), 
\]%
where $D(M) $ is an appropriate constant depending on $M$, such that $D(M) \rightarrow 0$ when $M\rightarrow 0$.
 Now, we can write%
\[
\frac{d}{dt}(\alpha |\w|^{2}+|\z|^{2})\leq 2 (D(M) -1)
(\nu |\nabla \w|^{2}+\chi |\nabla \z|^{2}). 
\]%
Thus, considering that $D\left( M\right) <1,$ we conclude that $L=2(
1-D(M)) >0,$ and then, from the above inequality, we have%
\begin{equation}
\frac{d}{dt}(\alpha |\w|^{2}+|\z|^{2})\leq -L(\nu |\nabla \w|^{2}+\chi
|\nabla \z|^{2}) .  \label{e1}
\end{equation}%

On the other hand, recall that we can choose the basis $\left\{
\mbox{$\w$}_{i};i=1,2,...\right\} $ such that the eigenfunctions $\mbox{$\w$}_{i}$ of $A$ are
also eigenfunctions of $A^{\gamma }$ and that we can write%
\[
A\mbox{$\w$}_{i}=\mu _{i}\mbox{$\w$}_{i},\qquad A^{\gamma }\mbox{$\w$}_{i}=\mu _{i}^{\gamma }\mbox{$\w$}_{i} 
\]%
where the $\mu _{i}$ are eigenvalue of $A.$ We obtain that%
\[
|\nabla \w|\leq \mu ^{1/2}\left\vert \w\right\vert \qquad \mbox{and\qquad }%
|\nabla \z|\leq \mu ^{1/2}\left\vert \z\right\vert , 
\]%
then from (\ref{e1}) we can write%
\begin{eqnarray*}
\frac{d}{dt}(\alpha |\w|^{2}+|\z|^{2}) &\leq &-L( \nu \mu |\w|^{2}+\chi
\mu |\z|^{2}) \\
&\leq &-Q(\alpha |\w|^{2}+|\z|^{2}),
\end{eqnarray*}%
where $Q=L\mu \min \left\{ \nu ,\chi \right\} \left( \frac{1}{\alpha }%
+1\right) >0.$

Finally, 
\[
(\alpha |\w(t) |^{2}+|\z(t) |^{2}\leq (\alpha
|\w(0)|^{2}+|\z(0)|^{2}) e^{-Qt}, 
\]%
for any $t\in (0,\infty ).$

Since $\w(t)$ and $\z(t)$ are periodic in $t$, for any $t\in (-\infty ,+\infty
)$ there exists a positive integer $n_{0}$ such that $t+n_{0}\tau >0$ and 
\[
\alpha |\w(t)|^{2}+|\z(t)|^{2}=\alpha |\w(t+n_{0}\tau )|^{2}+|\z(t+n_{0}\tau
)|^{2}. 
\]

Hence, it follows, 
\[
\alpha |\w(t)|^{2}+|\z(t)|^{2}\leq (\alpha |\w(0)|^{2}+|\z(0)|^{2}) e^{-Qnt} 
\]%
$(n\geq n_{0}),$ which implies 
\[
\alpha |\w(t)|^{2}+|\z(t)|^{2}=0 
\]%
and finally $\u_{1}=\u_{2}$ and $\h_{1}=\h
_{2}.$   Thus, Theorem 6 is proven.

\section{Asymptotic stability}

In this section we prove the theorem of  stability, for the two-dimensional case, by using the method of  Ref. \cite{Hsia} 
and comment on the proof for three-dimensional case.

\begin{proof}[(Proof of the Theorem 7]  Let $\left\{(\u_{2}(t) ,\h_{2}
(t)) \right\} _{t\geq 0}$ is a strong solution of the system (\ref{Ch1})-(\ref{Ch2}) 
with inhomogeneous conditions $\left( \u_{0},\h_{0}\right)$  which satisfies (\ref{S2}), 
and suppose  $\left\{ \left( \u_{1}\left( t\right),\h_{1}\left(t\right) \right) \right\} _{t\geq 0}$   is another strong solution. 
Let $\w=\u_{1}-\u_{2}$ and $\z=\h_{1}-\h_{2}$
then by substituting in the system \eqref{arica1}, we have%
\begin{equation}
\begin{array}{l}
\alpha \displaystyle{\frac{ d\w}{dt}}+\nu A\w+\alpha P%
\w\cdot \nabla \u_{1}+\alpha P\u_{2}\cdot \nabla 
\w-P\z\cdot \nabla \h_{1}-P\h_{2}\cdot
\nabla \z \\ 
\\ 
+P\w\cdot \nabla B_{1}+PB_{1}\cdot \nabla \w-PB_{2}\cdot
\nabla \z-P\z\cdot \nabla B_{2}=0,%
\end{array}
\label{S5}
\end{equation}

\begin{equation}
\begin{array}{l}
\displaystyle{\frac{d\z}{dt}} + \chi A\z+P\w
\cdot \nabla \h_{1}+P\u_{2}\cdot \nabla \z-P\z%
\cdot \nabla \u_{1}-P\h_{2}\cdot \nabla \w \\ 
\\ 
-P\z\cdot \nabla B_{1}+PB_{1}\cdot \nabla \z-PB_{2}\cdot
\nabla \w+P\w\cdot \nabla B_{2}=0.
\end{array}
\label{S6}
\end{equation}%
Now, taking the $L^{2}(\Omega)$ inner product of (\ref{S5})
with $A\w$, and observing that%
\[
\begin{array}{l}
\alpha \left( \w\cdot \nabla \u_{1},A\w\right)
=\alpha \left( \w\cdot \nabla \w,A\w\right) +\alpha
\left( \w\cdot \nabla \u_{2},A\w\right), \\ 
\\ 
\left( \z\cdot \nabla \h_{1},A\w\right) =\left( 
\z\cdot \nabla \z A\w\right) +\left( \z\cdot
\nabla \h_{2},A\w\right)%
\end{array}%
\]%
we have%
\begin{equation}
\begin{array}{lll}
\displaystyle{\frac{\alpha }{2}\frac{d}{dt}}\left\vert \nabla \w\right\vert
^{2}+\nu \left\vert A\w\right\vert ^{2} & = & -\alpha \left( \w\cdot \nabla \w,A\w\right) -\alpha \left( \w\cdot
\nabla \u_{2},A\w\right) \bigskip \\ 
&  & -\alpha \left( \u_{2}\cdot \nabla \w,A\w\right)
+\left( \z\cdot \nabla \z,A\w\right) \bigskip \\ 
&  & +\left( \z\cdot \nabla \h_{2},A\w\right)
+\left( \h_{2}\cdot \nabla \z,A\w\right) \bigskip \\ 
&  & -\left( \w\cdot \nabla B_{1},A\w\right) -\left(
B_{1}\cdot \nabla \w,A\w\right) \bigskip \\ 
&  & +\left( B_{2}\cdot \nabla \z,A\w\right) +\left( \z\cdot \nabla B_{2},A\w\right) .%
\end{array}
\label{S7}
\end{equation}

In the same way,\ taking the $L^{2}(\Omega) $ inner product of (%
\ref{S6}) with $A\z$, and observing that%
\[
\begin{array}{l}
(\w\cdot \nabla \h_{1},A\z) =(\w\cdot \nabla \z,A\z) +(\w\cdot
\nabla \h_{2},A\z) \\ 
\\ 
(\z\cdot \nabla \u_{1},A\z) =( 
\z\cdot \nabla \w,A\z) +(\z\cdot
\nabla \u_{2},A\z)%
\end{array}%
\]%
we have%
\begin{equation}
\begin{array}{lll}
\displaystyle{\frac{1}{2}\frac{d}{dt}}\left\vert \nabla \z\right\vert
^{2}+\chi \left\vert A\z\right\vert ^{2} & = & -(\w%
\cdot \nabla \z,A\z) -(\w\cdot \nabla 
\h_{2},A\z) \bigskip \\ 
&  & -(\u_{2}\cdot \nabla \z,A\z)
+(\z\cdot \nabla \w,A\z) \bigskip \\ 
&  & +(\z\cdot \nabla \u_{2},A\z)
+(\h_{2}\cdot \nabla \w,A\z) \bigskip \\ 
&  & +(\z\cdot \nabla B_{1},A\z) -(
B_{1}\cdot \nabla \z,A\z) \bigskip \\ 
&  & +( B_{2}\cdot \nabla \w,A\z) -(\w
\cdot \nabla B_{2},A\z) .%
\end{array}
\label{S8}
\end{equation}

Now, we must limit each term on the right side of equalities (\ref{S7}), 
\begin{eqnarray*}
\left\vert -\alpha (\w\cdot \nabla \w,A\w) \right\vert &\leq &\alpha \left\vert \w\right\vert
_{L^{4}}\left\vert \nabla \w\right\vert _{L^{4}}\left\vert A\w%
\right\vert \leq \alpha C_{\varepsilon }\left\vert \w\right\vert
_{L^{4}}^{2}\left\vert \nabla \w\right\vert _{L^{4}}^{2}+\alpha
\varepsilon \left\vert A\w\right\vert ^{2}\bigskip \\
&\leq &\alpha CC_{\varepsilon \varepsilon }\left\vert \w\right\vert
\left\vert \nabla \w\right\vert \left\vert \nabla \w%
\right\vert \left\vert A\w\right\vert +\alpha \varepsilon \left\vert
A\w\right\vert ^{2}\bigskip \\
&\leq &\alpha CC_{\varepsilon \delta }\left\vert \w\right\vert
^{2}\left\vert \nabla \w\right\vert ^{4}+\alpha C_{\varepsilon
}\delta \left\vert A\w\right\vert ^{2}+\alpha \varepsilon \left\vert
A\w\right\vert ^{2}\bigskip \\
&\leq &\alpha CC_{\varepsilon \delta }\left\vert \w\right\vert
^{2}\left\vert \nabla \w\right\vert ^{4}+\frac{\nu }{44}\left\vert A%
\w\right\vert ^{2}, ~~~ \left( \alpha C_{\varepsilon }\delta
+\alpha \varepsilon <\frac{\nu }{44}\right),
\end{eqnarray*}%
where we have used the fact $\u_{2}$ is a strong solution of the system (1)-(3), 
\begin{eqnarray*}
\left\vert -\alpha (\w\cdot \nabla \u_{2},A\w) \right\vert &\leq &\alpha C\left\vert \w\right\vert
_{H^{2}}\left\vert \nabla \u_{2}\right\vert \left\vert A\w%
\right\vert \bigskip \\
&\leq &\alpha C\left\vert \nabla \u_{2}\right\vert \left\vert A%
\w\right\vert ^{2}\leq C\left( \gamma _{1},\gamma _{2}\right)
\left\vert A\w\right\vert ^{2},
\end{eqnarray*}%
\begin{eqnarray*}
\left\vert -\alpha (\u_{2}\cdot \nabla \w,A\w) \right\vert &\leq &\alpha \left\vert \u_{2}\right\vert
_{L^{4}}\left\vert \nabla \w\right\vert _{L^{4}}\left\vert A\w%
\right\vert \leq \alpha C\left\vert \u_{2}\right\vert
^{1/2}\left\vert \nabla \u_{2}\right\vert ^{1/2}\left\vert \nabla 
\w\right\vert ^{1/2}\left\vert A\w\right\vert ^{3/2}\bigskip
\\
&\leq &\alpha CC_{\varepsilon }\left\vert \u_{2}\right\vert
^{2}\left\vert \nabla \u_{2}\right\vert ^{2}\left\vert \nabla 
\w\right\vert ^{2}+\frac{\nu }{44}\left\vert A\w\right\vert
^{2}, ~~~ \left( \alpha C\varepsilon <\frac{\nu }{44}\right),
\end{eqnarray*}%
\begin{eqnarray*}
\left\vert \left( \z\cdot \nabla \z,A\w\right)
\right\vert &\leq &\left\vert \z\right\vert _{L^{4}}\left\vert
\nabla \z\right\vert _{L^{4}}\left\vert A\w\right\vert \leq
C\left\vert \z\right\vert ^{1/2}\left\vert \nabla \z%
\right\vert ^{1/2}\left\vert \nabla \z\right\vert ^{1/2}\left\vert A%
\z\right\vert ^{1/2}\left\vert A\w\right\vert \bigskip \\
&\leq &C\left( C_{\varepsilon }\left\vert \z\right\vert \left\vert
\nabla \z\right\vert ^{2}\left\vert A\z\right\vert
+\varepsilon \left\vert A\w\right\vert ^{2}\right) \bigskip \\
&\leq &CC_{\varepsilon ,\delta }\left\vert \z\right\vert
^{2}\left\vert \nabla \z\right\vert ^{4}+\frac{\chi }{48}%
\left\vert A\z\right\vert ^{2}+\frac{\nu }{44}\left\vert A\w%
\right\vert ^{2}, ~~~~ \left( C\varepsilon <\frac{\nu }{44}\right),
\end{eqnarray*}%
\begin{eqnarray*}
\left\vert \left( \z\cdot \nabla \h_{2},A\w\right)
\right\vert &\leq &C\left\vert \z\right\vert _{H^{2}}\left\vert
\nabla \h_{2}\right\vert \left\vert A\w\right\vert \leq
C\left\vert \nabla \h_{2}\right\vert \left\vert A\z%
\right\vert \left\vert A\w\right\vert \bigskip \\
&\leq &\frac{C\left( \gamma _{1},\gamma _{2}\right) }{2}\left\vert A\z%
\right\vert ^{2}+\frac{C\left( \gamma _{1},\gamma _{2}\right) }{2}%
\left\vert A\w\right\vert ^{2},
\end{eqnarray*}%
\begin{eqnarray*}
\left\vert (\h_{2}\cdot \nabla \z,A\w)
\right\vert &\leq &\left\vert \h_{2}\right\vert _{L^{4}}\left\vert
\nabla \z\right\vert _{L^{4}}\left\vert A\w\right\vert \leq
C\left\vert \h_{2}\right\vert ^{1/2}\left\vert \nabla \h%
_{2}\right\vert ^{1/2}\left\vert \nabla \z\right\vert
^{1/2}\left\vert A\z\right\vert ^{1/2}\left\vert A\w%
\right\vert \bigskip \\
&\leq &CC_{\varepsilon }\left\vert \h_{2}\right\vert \left\vert
\nabla \h_{2}\right\vert \left\vert \nabla \z\right\vert
\left\vert A\z\right\vert +C\varepsilon \left\vert A\w%
\right\vert ^{2}\bigskip \\
&\leq &CC_{\varepsilon ,\delta }\left\vert \h_{2}\right\vert
^{2}\left\vert \nabla \h_{2}\right\vert ^{2}\left\vert \nabla 
\z\right\vert ^{2}+C\delta \left\vert A\z\right\vert
+C\varepsilon \left\vert A\w\right\vert ^{2}\bigskip \\
&\leq &CC\left( \gamma _{1},\gamma _{2}\right) C_{\varepsilon ,\delta
}\left\vert \nabla \z\right\vert ^{2}+\frac{\chi}{48}%
\left\vert A\z\right\vert +\frac{\nu }{44}\left\vert A\w%
\right\vert ^{2},
\end{eqnarray*}%
\[
\left\vert (\w\cdot \nabla B_{1},A\w)
\right\vert \leq C\left\vert \w\right\vert _{H^{2}}\left\vert \nabla
B_{1}\right\vert \left\vert A\w\right\vert \leq C\left( \gamma
_{1},\gamma _{2}\right) \left\vert A\w\right\vert ^{2}, \quad \left(
C\left\vert \nabla B_{1}\right\vert \leq C\left( \gamma _{1},\gamma
_{2}\right) \right),  
\]%
\begin{eqnarray*}
\left\vert (B_{1}\cdot \nabla \w,A\w)
\right\vert &\leq &\left\vert B_{1}\right\vert _{L^{4}}\left\vert \nabla 
\w\right\vert _{L^{4}}\left\vert A\w\right\vert \leq
C\left\vert B_{1}\right\vert ^{1/2}\left\vert \nabla B_{1}\right\vert
^{1/2}\left\vert \nabla \w\right\vert ^{1/2}\left\vert A\w%
\right\vert ^{3/2}\bigskip \\
&\leq &CC_{\varepsilon }\left\vert B_{1}\right\vert ^{2}\left\vert \nabla
B_{1}\right\vert ^{2}\left\vert \nabla \w\right\vert ^{2}+\frac{\nu 
}{44}\left\vert A\w\right\vert ^{2}, 
\end{eqnarray*}%
\begin{eqnarray*}
\left\vert (B_{2}\cdot \nabla \z,A\w)
\right\vert &\leq &\left\vert B_{2}\right\vert _{L^{4}}\left\vert \nabla 
\z\right\vert _{L^{4}}\left\vert A\w\right\vert \leq
C\left\vert B_{2}\right\vert ^{1/2}\left\vert \nabla B_{2}\right\vert
^{1/2}\left\vert \nabla \z\right\vert ^{1/2}\left\vert A\z%
\right\vert ^{1/2}\left\vert A\w\right\vert \bigskip \\
&\leq &CC_{\varepsilon }\left\vert B_{2}\right\vert \left\vert \nabla
B_{2}\right\vert \left\vert \nabla \z\right\vert \left\vert A\z\right\vert +C\varepsilon \left\vert A\w\right\vert ^{2}\bigskip \\
&\leq &CC_{\varepsilon ,\delta }\left\vert B_{2}\right\vert ^{2}\left\vert
\nabla B_{2}\right\vert ^{2}\left\vert \nabla \z\right\vert
^{2}+C\delta \left\vert A\z\right\vert +C\varepsilon \left\vert A%
\w\right\vert ^{2}\bigskip \\
&\leq &CC_{\varepsilon ,\delta }\left\vert B_{2}\right\vert ^{2}\left\vert
\nabla B_{2}\right\vert ^{2}\left\vert \nabla \z\right\vert ^{2}+%
\frac{\chi}{48}\left\vert A\z\right\vert +\frac{\nu }{44}%
\left\vert A\w\right\vert ^{2}, 
\end{eqnarray*}%
\begin{eqnarray*}
\left\vert (\z\cdot \nabla B_{2},A\w)
\right\vert &\leq &C\left\vert \z\right\vert _{H^{2}}\left\vert
\nabla B_{2}\right\vert \left\vert A\w\right\vert \leq C\left\vert
\nabla B_{2}\right\vert \left\vert A\z\right\vert \left\vert A%
\w\right\vert \bigskip \\
&\leq &C\left\vert \nabla B_{2}\right\vert \left( C_{\varepsilon }\left\vert
A\z\right\vert ^{2}+\varepsilon \left\vert A\w\right\vert
^{2}\right) \bigskip \\
&\leq &\frac{\chi}{48}\left\vert A\z\right\vert ^{2}+\frac{\nu 
}{44}\left\vert A\w\right\vert ^{2}.
\end{eqnarray*}%
Now, we must limit each term on the right side of equalities (\ref{S8}),%
\begin{eqnarray*}
\left\vert -(\w\cdot \nabla \z,A\z)
\right\vert &\leq &\left\vert \w\right\vert _{L^{4}}\left\vert
\nabla \z\right\vert _{L^{4}}\left\vert A\z\right\vert \leq
C\left\vert \w\right\vert ^{1/2}\left\vert \nabla \w%
\right\vert ^{1/2}\left\vert \nabla \z\right\vert ^{1/2}\left\vert A%
\z\right\vert ^{3/2}\bigskip \\
&\leq &C_{\delta }\left\vert \w\right\vert ^{2}\left\vert \nabla 
\w\right\vert ^{2}\left\vert \nabla \z\right\vert
^{2}+\delta \left\vert A\z\right\vert ^{2}\bigskip \\
&\leq &C_{\delta \tau }\left\vert \w\right\vert ^{4}\left\vert
\nabla \w\right\vert ^{4}+\tau \left\vert \nabla \z%
\right\vert ^{4}+\frac{\chi}{48}\left\vert A\z\right\vert ^{2},
\end{eqnarray*}%
\begin{eqnarray*}
\left\vert -(\w\cdot \nabla \h_{2},A\z)
\right\vert &\leq &C\left\vert \w\right\vert _{H^{2}}\left\vert
\nabla \h_{2}\right\vert \left\vert A\z\right\vert \leq
C\left\vert \nabla \h_{2}\right\vert \left\vert A\w%
\right\vert \left\vert A\z\right\vert \bigskip \\
&\leq &\frac{C\left( \gamma _{1},\gamma _{2}\right) }{2}\left\vert A\w%
\right\vert ^{2}+\frac{C\left( \gamma _{1},\gamma _{2}\right) }{2}%
\left\vert A\z\right\vert ^{2},
\end{eqnarray*}%
\[
\left\vert (\u_{2}\cdot \nabla \z,A\z)
\right\vert \leq CC_{\delta }\left\vert \u_{2}\right\vert
^{2}\left\vert \nabla \u_{2}\right\vert ^{2}\left\vert \nabla 
\z\right\vert ^{2}+\frac{\chi}{48}\left\vert A\z%
\right\vert ^{2}, 
\]%
\begin{eqnarray*}
\left\vert (\z\cdot \nabla \w,A\z)
\right\vert &\leq &\left\vert \z\right\vert _{L^{4}}\left\vert
\nabla \w\right\vert _{L^{4}}\left\vert A\z\right\vert \leq
C\left\vert \z\right\vert ^{1/2}\left\vert \nabla \z%
\right\vert ^{1/2}\left\vert \nabla \w\right\vert ^{1/2}\left\vert A%
\w\right\vert ^{1/2}\left\vert A\z\right\vert \bigskip \\
&\leq &C_{\delta ,\varepsilon ,\lambda }\left\vert \z\right\vert
^{4}\left\vert \nabla \z\right\vert ^{4}+\lambda \left\vert \nabla 
\w\right\vert ^{4}+\frac{\nu }{44}\left\vert A\w\right\vert
^{2}+\frac{\chi}{48}\left\vert A\z\right\vert ^{2}, 
\end{eqnarray*}%
\begin{eqnarray*}
\left\vert (\z\cdot \nabla \u_{2},A\z)
\right\vert &\leq &C\left\vert \z\right\vert _{H^{2}}\left\vert
\nabla \u_{2}\right\vert \left\vert A\z\right\vert \leq
C\left\vert \nabla \u_{2}\right\vert \left[ \left\vert \nabla 
\z\right\vert +\left\vert A\z\right\vert \right] \left\vert A%
\z\right\vert \bigskip \\
&\leq &C\left\vert \nabla \u_{2}\right\vert \left\vert \nabla 
\z\right\vert \left\vert A\z\right\vert +C\left\vert \nabla 
\u_{2}\right\vert \left\vert A\z\right\vert ^{2}\bigskip \\
&\leq &CC_{\delta }\left\vert \nabla \u_{2}\right\vert
^{2}\left\vert \nabla \z\right\vert ^{2}+\frac{\chi}{48}%
\left\vert A\z\right\vert ^{2}+C\left( \gamma _{1},\gamma
_{2}\right) \left\vert A\z\right\vert ^{2}, 
\end{eqnarray*}

\begin{eqnarray*}
\left\vert (\h_{2}\cdot \nabla \w,A\z)
\right\vert  &\leq &\left\vert \h_{2}\right\vert _{L^{4}}\left\vert
\nabla \w\right\vert _{L^{4}}\left\vert A\z\right\vert \leq
C\left\vert \h_{2}\right\vert ^{1/2}\left\vert \nabla \h%
_{2}\right\vert ^{1/2}\left\vert \nabla \w\right\vert
^{1/2}\left\vert A\w\right\vert ^{1/2}\left\vert A\z%
\right\vert \bigskip  \\
&\leq &C_{\delta }\left\vert \h_{2}\right\vert \left\vert \nabla 
\h_{2}\right\vert \left\vert \nabla \w\right\vert \left\vert
A\w\right\vert +\delta \left\vert A\z\right\vert
^{2}\bigskip  \\
&\leq &C_{\delta ,\varepsilon }\left\vert \h_{2}\right\vert
^{2}\left\vert \nabla \h_{2}\right\vert ^{2}\left\vert \nabla 
\w\right\vert ^{2}+\frac{\nu }{44}\left\vert A\w\right\vert
^{2}+\frac{\chi}{48}\left\vert A\z\right\vert ^{2},
\end{eqnarray*}%
\begin{eqnarray*}
\left\vert \left( B_{1}\cdot \nabla \z,A\z\right)
\right\vert  &\leq &CC_{\varepsilon }\left\vert B_{1}\right\vert
^{2}\left\vert \nabla B_{1}\right\vert ^{2}\left\vert \nabla \z%
\right\vert ^{2}+\frac{\chi}{44}\left\vert A\z\right\vert
^{2}\bigskip  \\
&\leq &CC_{\varepsilon }\left\vert \nabla \z\right\vert ^{2}+\frac{%
\varkappa }{48}\left\vert A\z\right\vert ^{2}, \quad \left(
CC_{\varepsilon }\left\vert B_{1}\right\vert ^{2}\left\vert \nabla
B_{1}\right\vert ^{2}\leq CC_{\varepsilon }\right),
\end{eqnarray*}%
\begin{eqnarray*}
\left\vert (\z\cdot \nabla B_{1},A\z)
\right\vert  &\leq &C\left\vert \z\right\vert _{H^{2}}\left\vert
\nabla B_{1}\right\vert \left\vert A\z\right\vert \leq C\left\vert
\nabla B_{1}\right\vert \left[ \left\vert \nabla \z\right\vert
+\left\vert A\z\right\vert \right] \left\vert A\z\right\vert
\bigskip  \\
&\leq &C\left\vert \nabla B_{1}\right\vert \left\vert \nabla \z%
\right\vert \left\vert A\z\right\vert +C\left\vert \nabla
B_{1}\right\vert \left\vert A\z\right\vert ^{2}\bigskip  \\
&\leq &CC_{\delta }\left\vert \nabla B_{1}\right\vert ^{2}\left\vert \nabla 
\z\right\vert ^{2}+C\delta \left\vert A\z\right\vert
^{2}+C\left\vert \nabla B_{1}\right\vert \left\vert A\z\right\vert
^{2}\bigskip  \\
&\leq &CC_{\delta }\left\vert \nabla \z\right\vert ^{2}+\frac{%
\chi}{48}\left\vert A\z\right\vert ^{2}+C\left( \gamma
_{1},\gamma _{2}\right) \left\vert A\z\right\vert ^{2},\quad \left(
C\left\vert \nabla B_{1}\right\vert \leq C\left( \gamma _{1},\gamma
_{2}\right) \right), 
\end{eqnarray*}%
\begin{eqnarray*}
\left\vert (\w\cdot \nabla B_{2},A\z)
\right\vert  &\leq &C\left\vert \w\right\vert _{H^{2}}\left\vert
\nabla B_{2}\right\vert \left\vert A\z\right\vert \leq C\left\vert
\nabla B_{2}\right\vert \left\vert A\w\right\vert \left\vert A%
\z\right\vert \bigskip  \\
&\leq &C\left\vert \nabla B_{2}\right\vert \varepsilon \left\vert A\w%
\right\vert ^{2}+C\left\vert \nabla B_{2}\right\vert C_{\delta }\left\vert A%
\z\right\vert ^{2}\bigskip  \\
&\leq &\frac{\nu }{44}\left\vert A\w\right\vert ^{2}+C\left( \gamma
_{1},\gamma _{2}\right) \left\vert A\z\right\vert ^{2}, \quad \left(
C\left\vert \nabla B_{2}\right\vert C_{\delta }\leq C\left( \gamma
_{1},\gamma _{2}\right) \right),  
\end{eqnarray*}%
\begin{eqnarray*}
\left\vert \left( B_{2}\cdot \nabla \w,A\z\right)
\right\vert  &\leq &\left\vert B_{2}\right\vert _{L^{4}}\left\vert \nabla 
\w\right\vert _{L^{4}}\left\vert A\z\right\vert \leq
C\left\vert B_{2}\right\vert ^{1/2}\left\vert \nabla B_{2}\right\vert
^{1/2}\left\vert \nabla \w\right\vert ^{1/2}\left\vert A\w%
\right\vert ^{1/2}\left\vert A\z\right\vert \bigskip  \\
&\leq &C_{\delta }\left\vert B_{2}\right\vert \left\vert \nabla
B_{2}\right\vert \left\vert \nabla \w\right\vert \left\vert A\w\right\vert +\delta \left\vert A\z\right\vert ^{2}\bigskip  \\
&\leq &CC_{\delta ,\varepsilon }\left\vert \nabla \w\right\vert ^{2}+%
\frac{\nu }{44}\left\vert A\w\right\vert ^{2}+\frac{\chi}{48}%
\left\vert A\z\right\vert ^{2}.
\end{eqnarray*}%
Adding  equalities (\ref{S7}) and (\ref{S8}), from the previous estimates
we obtain%
\[
\begin{array}{l}
\displaystyle{\frac{d}{dt}}\left( \alpha \left\vert \nabla \w\right\vert
^{2}+\left\vert \nabla \z\right\vert ^{2}\right) +\left( \frac{3}{2}%
\nu -6C\left( \gamma _{1},\gamma _{2}\right) \right) \left\vert A\w%
\right\vert ^{2}+\left( \frac{3}{2}\chi -8C\left( \gamma _{1},\gamma
_{2}\right) \right) \left\vert A\z\right\vert ^{2}\bigskip  \\ 
\leq 2\alpha CC_{\varepsilon \delta }\left\vert \w\right\vert
^{2}\left\vert \nabla \w\right\vert ^{4}+2\alpha CC_{\varepsilon
}\left\vert \u_{2}\right\vert ^{2}C\left( \gamma _{1},\gamma
_{2}\right) \left\vert \nabla \w\right\vert ^{2}+2CC_{\varepsilon
,\delta }\left\vert \z\right\vert ^{2}\left\vert \nabla \z%
\right\vert ^{4}\bigskip  \\ 
+2CC\left( \gamma _{1},\gamma _{2}\right) C_{\varepsilon ,\delta }\left\vert
\nabla \z\right\vert ^{2}+2C_{\varepsilon }\left\vert
B_{1}\right\vert ^{2}\left\vert \nabla B_{1}\right\vert ^{2}\left\vert
\nabla \w\right\vert ^{2}+2CC_{\varepsilon ,\delta }\left\vert
B_{2}\right\vert ^{2}\left\vert \nabla B_{2}\right\vert ^{2}\left\vert
\nabla \z\right\vert ^{2}\bigskip  \\ 
+2C_{\delta \tau }\left\vert \w\right\vert ^{4}\left\vert \nabla 
\w\right\vert ^{4}+2C_{\varepsilon }\left\vert \u%
_{2}\right\vert ^{2}C\left( \gamma _{1},\gamma _{2}\right) \left\vert \nabla 
\z\right\vert ^{2}+2C_{\delta ,\varepsilon ,\lambda }\left\vert 
\z\right\vert ^{4}\left\vert \nabla \z\right\vert
^{4}\bigskip  \\ 
+2\lambda \left\vert \nabla \w\right\vert ^{4}+2CC_{\delta }C\left(
\gamma _{1},\gamma _{2}\right) \left\vert \nabla \z\right\vert
^{2}+2C_{\delta ,\varepsilon }\left\vert \h_{2}\right\vert
^{2}C\left( \gamma _{1},\gamma _{2}\right) \left\vert \nabla \w%
\right\vert ^{2}\bigskip  \\ 
+2CC_{\varepsilon }\left\vert B_{1}\right\vert ^{2}\left\vert \nabla
B_{1}\right\vert ^{2}\left\vert \nabla \z\right\vert
^{2}+2CC_{\delta }\left\vert \nabla B_{1}\right\vert ^{2}\left\vert \nabla 
\mathbf{z}\right\vert ^{2}+2C_{\delta ,\varepsilon }\left\vert
B_{2}\right\vert ^{2}\left\vert \nabla B_{2}\right\vert ^{2}\left\vert
\nabla \w\right\vert ^{2}.%
\end{array}%
\]%
Let 
\[
\Pi =2\max \left\{ 
\begin{array}{l}
\alpha CC_{\varepsilon \delta },\alpha CC_{\varepsilon }\left\vert \u%
_{2}\right\vert ^{2},CC_{\varepsilon ,\delta },CC_{\varepsilon ,\delta
},C_{\varepsilon }\left\vert B_{1}\right\vert ^{2}\left\vert \nabla
B_{1}\right\vert ^{2},\bigskip  \\ 
CC_{\varepsilon ,\delta }\left\vert B_{2}\right\vert ^{2}\left\vert \nabla
B_{2}\right\vert ^{2},C_{\delta \tau },C_{\varepsilon }\left\vert \u%
_{2}\right\vert ^{2},C_{\delta ,\varepsilon ,\lambda },\lambda ,CC_{\delta
},\bigskip  \\ 
C_{\delta ,\varepsilon }\left\vert \h_{2}\right\vert
^{2},CC_{\varepsilon }\left\vert B_{1}\right\vert ^{2}\left\vert \nabla
B_{1}\right\vert ^{2},CC_{\delta }\left\vert \nabla B_{1}\right\vert
^{2},C_{\delta ,\varepsilon }\left\vert B_{2}\right\vert ^{2}\left\vert
\nabla B_{2}\right\vert ^{2}%
\end{array}%
\right\} 
\]%
\[
\]%
\begin{equation}
\begin{array}{l}
\displaystyle{\frac{d}{dt}}\left( \alpha \left\vert \nabla \w\right\vert
^{2}+\left\vert \nabla \z\right\vert ^{2}\right) +\left( \frac{3}{2}%
\nu -6C\left( \gamma _{1},\gamma _{2}\right) \right) \left\vert A\w%
\right\vert ^{2}+\left( \frac{3}{2}\chi -8C\left( \gamma _{1},\gamma
_{2}\right) \right) \left\vert A\z\right\vert ^{2}\bigskip  \\ 
\leq \Pi \left\{ \left[ \left\vert \w\right\vert ^{2}\left\vert
\nabla \w\right\vert ^{2}+2C\left( \gamma _{1},\gamma _{2}\right)
+2+\left\vert \w\right\vert ^{4}\left\vert \nabla \w%
\right\vert ^{2}+\left\vert \nabla \w\right\vert ^{2}\right]
\left\vert \nabla \mathbf{w}\right\vert ^{2}\right. \bigskip  \\ 
\left. +\left[ \left\vert \z\right\vert ^{2}\left\vert \nabla 
\z\right\vert ^{2}+3C\left( \gamma _{1},\gamma _{2}\right)
+3+\left\vert \z\right\vert ^{4}\left\vert \nabla \z%
\right\vert ^{2}\right] \left\vert \nabla \z\right\vert ^{2}\right\}
.%
\end{array}
\label{S9}
\end{equation}%
Now, we can choose $\gamma _{1}$ and $\gamma _{2}$ small, so that the
following inequalities hold, 
\[
C\left( \gamma _{1},\gamma _{2}\right) <\frac{\nu }{12}~~~\rm{ and }~~~C\left(
\gamma _{1},\gamma _{2}\right) <\frac{\chi }{16},
\]%
then, from inequality (\ref{S9}) we get,%
\[
\begin{array}{l}
\displaystyle{\frac{d}{dt}}\left( \alpha \left\vert \nabla \w\right\vert
^{2}+\left\vert \nabla \z\right\vert ^{2}\right) +\nu \left\vert A%
\w\right\vert ^{2}+\chi \left\vert A\z\right\vert
^{2}\bigskip  \\ 
\leq \Pi \left\{ \left[ \displaystyle{\frac{1}{\alpha }}\left( 1+\left\vert \w%
\right\vert ^{2}+\left\vert \w\right\vert ^{4}\right) \left\vert
\nabla \w\right\vert ^{2}+\displaystyle{\frac{2}{\alpha }}C\left( \gamma
_{1},\gamma _{2}\right) +\frac{2}{\alpha }\right] \alpha \left\vert \nabla 
\w\right\vert ^{2}\right. \bigskip  \\ 
\left. +\left[ \left( \left\vert \z\right\vert ^{2}+\left\vert 
\z\right\vert ^{4}\right) \left\vert \nabla \z\right\vert
^{2}+3C\left( \gamma _{1},\gamma _{2}\right) +3\right] \left\vert \nabla 
\z\right\vert ^{2}\right\},%
\end{array}%
\]%
or%
\begin{equation}
\begin{array}{c}
\displaystyle{\frac{d}{dt}}\left( \alpha \left\vert \nabla \w\right\vert
^{2}+\left\vert \nabla \z\right\vert ^{2}\right) +\nu \left\vert A%
\w\right\vert ^{2}+\chi \left\vert A\z\right\vert
^{2}\bigskip  \\ 
\leq \Pi P(t) \left( \alpha \left\vert \nabla \w%
\right\vert ^{2}+\left\vert \nabla \z\right\vert ^{2}\right),
\end{array}
\label{SS9}
\end{equation}%
where%
\[
P(t) =\frac{1}{\alpha }\left( 1+\left\vert \w%
\right\vert ^{2}+\left\vert \w\right\vert ^{4}\right) \left\vert
\nabla \w\right\vert ^{2}+\left( \left\vert \z\right\vert
^{2}+\left\vert \z\right\vert ^{4}\right) \left\vert \nabla \z%
\right\vert ^{2}+\left( \frac{2}{\alpha }+3\right) \left( C\left( \gamma
_{1},\gamma _{2}\right) +1\right) .
\]%
Then, from (\ref{SS9}) and (\ref{e1}) we have%
\begin{equation}
\displaystyle{\frac{d}{dt}}\left( \alpha \left\vert \w\right\vert ^{2}+\left\vert 
\z\right\vert ^{2}\right) +L\left( \nu \left\vert \nabla \w%
\right\vert ^{2}+\chi \left\vert \nabla \z\right\vert ^{2}\right)
\leq 0,   \label{S10}
\end{equation}%
\begin{equation}
\begin{array}{c}
\displaystyle{\frac{d}{dt}}\left( \alpha \left\vert \nabla \w\right\vert
^{2}+\left\vert \nabla \z\right\vert ^{2}\right) +\nu \left\vert A%
\w\right\vert ^{2}+\chi \left\vert A\z\right\vert
^{2}\bigskip  \\ 
\leq \Pi P\left( t\right) \left( \alpha \left\vert \nabla \w%
\right\vert ^{2}+\left\vert \nabla \z\right\vert ^{2}\right).
\end{array}
\label{S11}
\end{equation}%
Note that from (\ref{S10}) we can infer that%
\begin{equation}
\alpha \left\vert \w(t) \right\vert ^{2}+\left\vert 
\z(t) \right\vert ^{2}\leq e^{-\beta Lt}\left( \alpha
\left\vert \w(0) \right\vert ^{2}+\left\vert \z%
(0) \right\vert ^{2}\right) ,\quad \forall t\geq 0,  \label{S12}
\end{equation}%
where $\beta =\min \left\{ \displaystyle{\frac{\nu }{\alpha }},\chi \right\} .$

Now, to derive bound for $\alpha \left\vert \nabla \w\right\vert
^{2}+\left\vert \nabla \z\right\vert ^{2},$ we take $g(
t) =\alpha \left\vert \nabla \w(t) \right\vert
^{2}+\left\vert \nabla \z(t) \right\vert ^{2}$ and
rewrite (\ref{S11}) as%
\begin{equation}
g^{\prime }\left( t\right) \leq \Pi P\left( t\right) g\left( t\right) .
\label{S13}
\end{equation}%

Now for any positive $t_{1}>0,$ by integrating (\ref{S10}) over the
interval $\left[ t_{1},t_{1}+1\right] ,$ we obtain that%
\begin{equation}
L \beta \int_{t_{1}}^{t_{1}+1} \left( \alpha \left\vert \nabla \w(s)
\right\vert ^{2}+\left\vert \nabla \z(s) \right\vert
^{2} \right) ds\leq \alpha \left\vert \w (t_{1}) \right\vert
^{2}+\left\vert \z(t_{1}) \right\vert ^{2}.  \label{S14}
\end{equation}

By mean value theorem, there exists a number $t_{0}\in \left[ t_{1},t_{1}+1%
\right] $ such that%
\begin{equation}
L \beta \left(  \alpha \left\vert \nabla \w(t_{0}) \right\vert
^{2}+\left\vert \nabla \z(t_{0}) \right\vert ^{2} \right) \leq
\alpha \left\vert \w(t_{1}) \right\vert ^{2}+\left\vert 
\z(t_{1}) \right\vert ^{2}\leq e^{-\beta Lt_1}\left(
\alpha \left\vert \w(0) \right\vert ^{2}+\left\vert 
\z(0) \right\vert ^{2}\right) .  \label{S15}
\end{equation}%
Next, for any $0<\delta \leq 1,$ the integration of (\ref{S13}) over the
interval $\left[ t_{0},t_{0}+\delta \right] $ we obtain%
\begin{equation}
g(t_{0}+\delta) \leq e^{\int_{t_{0}}^{t_{0}+\delta}\Pi
P\left( s\right) ds}g(t_{0}) \leq \left(L \beta \right)^{-1} e^{-\beta Lt_{1}}\left( \alpha
\left\vert \w(0) \right\vert ^{2}+\left\vert \z%
(0) \right\vert ^{2}\right)
e^{\int_{t_{0}}^{t_{0}+1}\Pi P(s) ds}.  \label{S16}
\end{equation}%
Note that%
\begin{equation}
\begin{array}{l}
\displaystyle\int_{t_{0}}^{t_{0}+1}\Pi P\left( s\right) ds=\Pi
\int_{t_{0}}^{t_{0}+1}P\left( s\right) ds\bigskip  \\ 
=\Pi \displaystyle \int_{t_{0}}^{t_{0}+1}\left[ \displaystyle{\frac{1}{\alpha }}\left(
1+\left\vert \w\right\vert ^{2}+\left\vert \w\right\vert
^{4}\right) \left\vert \nabla \w\right\vert ^{2}+\left( \left\vert 
\z\right\vert ^{2}+\left\vert \z\right\vert ^{4}\right)
\left\vert \nabla \z\right\vert ^{2} \right] ds \bigskip  \\ 
\displaystyle + \Pi \int_{t_0}^{t_0+1} \left[ \left( \displaystyle{\frac{2}{\alpha }}+3\right) \left( C\left( \gamma _{1},\gamma
_{2}\right) +1\right) \right] ds  ,%
\end{array}
\label{S17}
\end{equation}%
then by (\ref{S12}) and (\ref{S14}) each term of the above  integral is bound
and not depend on the choice of $t_{1},t_{0}$ and $\delta .$ Hence, we infer
from (\ref{S16}) that there exist a constant $c_{1}$ independent of $t_{1}$
and $t_{0}$ such that%
\[
g\left( t_{1}+1\right) =g\left( t_{0}+\left( t_{1}+1-t_{0}\right) \right)
\leq c_{1}e^{-\beta Lt_1},
\]%
which implies that%
\[
\alpha \left\vert \nabla \w(t) \right\vert
^{2}+\left\vert \nabla \z(t) \right\vert ^{2}\leq
c_{1}e^{-\beta L(t-1) },
\]%
for any $t>1.$ Thus, the proof of theorem is complete. 
\end{proof}

{\bf Remark:} In this proof in order to estimate some terms, for example the term $%
\left\vert -\alpha \left( \mathbf{w}\cdot \nabla \mathbf{w},A\mathbf{w}%
\right) \right\vert ,$ we use the following Sobolev and Ladyzhenskaya's
inequality to $\varphi \in H^{1},$%
\[
\left\vert \varphi \right\vert _{L^{4}}\leq C\left\vert \varphi \right\vert
_{L^{2}}^{1/2}\left\vert \nabla \varphi \right\vert _{L^{2}}^{1/2},
\]%
where $C$ is a constant depending on the size of the domain,\ which is valid
for the two-dimensional case. The three-dimensional case is similar, but we
would have to use the inequality 
\[
\left\vert \varphi \right\vert _{L^{4}}\leq C\left\vert \varphi \right\vert
_{L^{2}}^{1/4}\left\vert \nabla \varphi \right\vert _{L^{2}}^{3/4},
\]%
however, this three-dimensional case will not be done in this work.

Now, we prove Theorem 9 on stability.

\begin{proof}
Let $(\u_{0}, \h_{0}) \in V\times V$ and $\F, \G\in H^{1}(\tau ;H)$ $(\tau >0).$ We assume that $( 
\u_{0}, \h_{0}) $ and $\F, \G$
satisfy the following conditions%
\[
\begin{array}{l}
\sup_{0\leq t\leq \tau }|\F|_{L^{\frac{N}{2}}(\Omega )}+ \sup_{0\leq t\leq \tau }|\G|_{L^{\frac{N}{2}}(\Omega )}\leq
M,\bigskip  \\ 
\sup_{0\leq t\leq \tau }|\nabla \u_0(t)|^{2}\leq
C(M_{0},M),\bigskip \\ 
\sup_{0\leq t\leq \tau }|\nabla \h_0(t)|^{2}\leq C(M_{0},M).
\end{array}%
\]%
Now, we denote by $(\u_{1}(
x,y,z,t), \h_{1}(x,y,z,t)) $ the solution
to the system (\ref{Ch1}) with the initial condition $(\u_{0}%
, \h_{0}), $ which is possible by theorem 7. 
 
Now, we should show that the sequences $\lbrace \boldsymbol{u}_{1}^{n} \rbrace$ and $\lbrace \boldsymbol{h}_{1}^{n} \rbrace$ given by 
\[
\u_{1}^{n}(x,y,z) \equiv \u_{1}(
x,y,z,n\tau) ;~~~~~\h_{1}^{n}(x,y,z) =%
\h_{1}(x,y,z,n\tau) .
\]%
are Cauchy sequences in $L^2(\Omega)$. In fact, because of the periodicity of the solutions for positive integers $m>k$, we can write a strong solution 
of the system \eqref{Ch1} 
\begin{eqnarray*}
\u_{2}(x,y,z,t)  &\equiv &\u_{1}(
x,y,z,t+( m-k) \tau), \bigskip  \\
\h_{2}(x,y,z,t)  &\equiv &\h_{1}(x,y,z,t+(
m-k) \tau) ,
\end{eqnarray*}%
with the initial condition $(\u_{2}(
x,y,z,0),\h_{2}(x,y,z,0)).$ 
\\
Moreover, we can see that%
\[
\begin{array}{lll}
\mathbf{\theta }\left( x,y,z,t\right)  & = & \u_{1}(
x,y,z,t) -\u_{2}(x,y,z,t), \bigskip  \\ 
\mathbf{\xi }\left( x,y,z,t\right)  & = & \h_{1}(
x,y,z,t) -\h_{2}(x,y,z,t) 
\end{array}%
\]%
satisfy the system (\ref{LS1}). Hence, taking $t=k\tau ,$ we obtain from (%
\ref{S12}) that%
\[
\alpha |\mathbf{\theta }(t)|^{2}+|\mathbf{\xi }(t)|^{2}\leq (\alpha |\mathbf{%
\theta }(0)|^{2}+|\mathbf{\xi }(0)|^{2})\exp (-\beta L k \tau)
\]%
or%
\[
\alpha |\u_{1}\left( k\tau \right) -\u_{1}\left( m\tau
\right) |^{2}+|\h_{1}\left( k\tau \right) -\h_{1}\left(
m\tau \right) |^{2}\leq (\alpha |\mathbf{\theta }(0)|^{2}+|\mathbf{\xi }%
(0)|^{2}) e^{(-\beta L k \tau)} ,
\]%
but under the hypotheses
\[
\alpha |\mathbf{\theta }(0)|^{2}+|\mathbf{\xi }(0)|^{2}\leq 2C\left(
M_{0},M\right),
\]%
thus, we deduce that the sequences $\left\{ \u_{1}^{n}\right\}
_{n\in \mathbb{N}}$ and $\left\{ \h_{1}^{n}\right\} _{n\in \mathbb{N}%
}$ are Cauchy sequences in $\L^{2}(\Omega).$

Now, let $\u_{1}(x,y,z) $ and $\h_{1}(
x,y,z) $ be the $L^{2}$ limit of $\left\{\u_{1}^{n}\right\}
_{n\in \mathbb{N}}$ and $\left\{ \mathbf{h}_{1}^{n}\right\} _{n\in \mathbb{N}%
}$ respectively. On the other hand, we know that
\[
\sup_{0\leq t\leq \tau }|\u_{1}^{n}|_{H^{1}}^{2}\leq C(M_{0},M)~~~\rm{
and }~~~\sup_{0\leq t\leq \tau }|\h_{1}^{n}|_{H^{1}}^{2}\leq C(M_{0},M).
\]%
Thus, we obtain subsequences $\left\{ n_{j}\right\} _{j\in \mathbb{N}}
$ and $\left\{ n_{l}\right\} _{l\in \mathbb{N}}$ of $\mathbb{N}$\ such that 
\[
\nabla \u_{1}^{n_{j}}\rightharpoonup \nabla \u_{1}~~~{ and \,}~~~\nabla \h_{1}^{n_{l}}\rightharpoonup \nabla 
\h_{1}~~~~\rm{ in }~~~~\L^{2}(\Omega)~~~\rm{ weakly.}~~~
\]%
Thus, $(\u_{1},\h_{1}) \in V\times V$ and
satisfy%
\[
|\u_{1}|_{H^{1}}^{2}\leq C(M_{0},M)~~~~\rm{ and }~~~|\h%
_{1}|_{H^{1}}^{2}\leq C(M_{0},M).
\]%
On the other hand, we denote by $(\u( x,y,z,t) ,\h( x,y,z,t)) $ the solution of system (\ref{Ch1}) with the
initial condition $(\u_{1},\h_{1})$ and we will show that this is time-periodic. In fact, let %
\[
\begin{array}{lll}
\mathbf{\theta }\left( x,y,z,t\right)  & = & \u(x,y,z,t)
-\u_{1}(x,y,z,t+n\tau) \bigskip  \\ 
\mathbf{\xi }(x,y,z,t)  & = & \h( x,y,z,t) -%
\h_{1}( x,y,z,t+n\tau) 
\end{array}%
\]%
and we observe that $\left( \mathbf{\theta },\mathbf{\xi }\right) $ satisfies the
system (\ref{LS1}). Then, by (\ref{S12}) we obtain%
\[
\alpha |\u(\tau) -\u_{1}^{n+1}|^{2}+|\h
(\tau) -\h_{1}^{n+1}|^{2}\leq (\alpha |\u_{1}-%
\u_{1}^{n}|^{2}+|\h_{1}-\h_{1}^{n}|^{2}) e^{(
-\beta L \tau)} ,
\]%
finally, taking the limit $n\rightarrow \infty ,$ we get%
\[
\alpha |\u(\tau) -\u(0) |^{2}+|%
\h(\tau) -\h(0) |^{2}=0.
\]
\end{proof}

\section{Navier-Stokes equation}

Note that the Navier-Stokes equations%
\[
\begin{array}{lll}
\displaystyle \frac{\partial \u}{\partial t}-\frac{\eta }{{\rho }}\Delta 
\u+\u\cdot \nabla \u & = & \f-\displaystyle\frac{1}{\rho }\nabla p^{\ast
},\bigskip \\ 
\u= \mbox{\boldmath $\beta$}(x,t) \mbox{ on }\partial \Omega,  &  & 
\bigskip \\ 
{\rm div}\,\u=0 &  & 
\end{array}%
\]%
are a particular case of the MHD equations when the magnetic field $\h$ is identically zero, in this case when $\h=0,$ we prove existence
and uniqueness of periodic strong solutions to the NS equations with
inhomogeneous boundary conditions. In Ref. \cite{Morimoto} Morimoto show existence and
uniqueness of weak solutions with inhomogeneous boundary conditions to the
NS equations. On the other hand, when the magnetic field $\h$ is
identically zero, we can reproduce the results on the asymptotic stability,
obtained by Hsia et al for the Navier-Stokes equations in \cite{Hsia}.

\subsection*{Acknowledgments}

We are grateful to all the three referees of this manuscript for their careful revision and many generous suggestions which improved it significantly.

\end{document}